\documentclass[11pt,reqno,oneside]{amsart}
\usepackage[T1]{fontenc}
\usepackage{mathpazo}
\usepackage[a4paper,left=30.5mm,right=30.5mm,textheight=224mm,footskip=24pt]{geometry}
\usepackage{amssymb,mathtools,microtype}
\usepackage{enumitem}
\usepackage{tikz}
\usetikzlibrary{arrows.meta}
\usepackage[hidelinks]{hyperref}
\hypersetup{pdftitle={A lifting defect for oriented t3 and bar t6 moves},pdfauthor={Wenbiao Liu}}
\newtheorem{theorem}{Theorem}[section]
\newtheorem{proposition}[theorem]{Proposition}
\newtheorem{lemma}[theorem]{Lemma}
\newtheorem{corollary}[theorem]{Corollary}
\newtheorem{conjecture}[theorem]{Conjecture}
\theoremstyle{definition}
\newtheorem{definition}[theorem]{Definition}
\newtheorem{example}[theorem]{Example}
\theoremstyle{remark}
\newtheorem{remark}[theorem]{Remark}
\theoremstyle{plain}
\newtheorem{thmA}{Theorem}

\newcommand{\F}{\mathbb F_3}
\newcommand{\R}{\mathcal R}

\newcommand{\Col}{\operatorname{Col}}
\newcommand{\Aff}{\operatorname{Aff}}
\newcommand{\Hom}{\operatorname{Hom}}
\newcommand{\rank}{\operatorname{rank}}
\newcommand{\im}{\operatorname{im}}
\newcommand{\rad}{\operatorname{rad}}

\newcommand{\eps}{\varepsilon}
\newcommand{\tr}{\mathsf T}

\newcommand{\coker}{\operatorname{coker}}

\numberwithin{equation}{section}
\newcommand{\twistcolumn}[2]{%
 \foreach \k in {1,...,#2}{%
  \pgfmathsetmacro{\ytop}{0.62*(#2/2-\k+1)}%
  \pgfmathsetmacro{\ybot}{\ytop-0.62}%
  \ifnum#1=1
    \draw[line width=.75pt] (0.34,\ytop)--(-0.34,\ybot);
    \draw[white,line width=3.5pt] (-0.34,\ytop)--(0.34,\ybot);
    \draw[line width=.75pt] (-0.34,\ytop)--(0.34,\ybot);
  \else
    \draw[line width=.75pt] (-0.34,\ytop)--(0.34,\ybot);
    \draw[white,line width=3.5pt] (0.34,\ytop)--(-0.34,\ybot);
    \draw[line width=.75pt] (0.34,\ytop)--(-0.34,\ybot);
  \fi}}

\title[A lifting defect for oriented moves]{A lifting defect for oriented $t_3$ and $\bar t_6$ moves}
\author{Wenbiao Liu}
\date{}
\begin{document}
\begin{abstract}
We study oriented \(t_3\)- and \(\bar t_6\)-moves through the problem of
lifting Fox \(3\)-colorings to colorings by a nine-element Alexander quandle.
We define a nonnegative integer lifting defect \(\delta\) for oriented
links and prove that it is invariant under these moves.
For knots, the defect equals the rank of an alternating bilinear form
obtained from a Seifert matrix, so it is even.
We also prove that it is additive under connected sum.
Let \(K_s\) be the connected sum of \(s\geq1\) copies of \(9_{46}\).
Then \(\delta(K_s)=2s\), so these knots represent infinitely many distinct
equivalence classes under the oriented moves, and none is equivalent
to an unlink. After forgetting orientation, each \(K_s\) can be transformed
into an unlink with \(2s+1\) components by \(3s\) ordinary \(3\)-moves.
\end{abstract}
\maketitle

\section{Introduction}\label{sec:intro}

A basic question in the study of local moves is the following:
given a collection of allowed local replacements, which knots or links
can be transformed into unlinks?
A \(3\)-move is an unoriented local move that replaces two parallel
crossing-free arcs with three half-twists, or performs the inverse
replacement.
In 1981, Nakanishi proposed the following conjecture, commonly known
as the Montesinos--Nakanishi \(3\)-move conjecture; see \cite{DP}.

\begin{conjecture}[Montesinos--Nakanishi]
	\label{conj:3move}
	Every unoriented link can be transformed into an unlink by a finite
	sequence of \(3\)-moves and  isotopies.
\end{conjecture}

For oriented links, local moves must also be compatible with the
orientations of the arcs.
We consider two oriented moves, \(t_3\) and \(\bar t_6\):
the former inserts or removes three half-twists between two coherently
oriented arcs, while the latter inserts or removes six half-twists
between two oppositely oriented arcs.
Twists of either handedness are allowed; see Figure~\ref{fig:moves}.

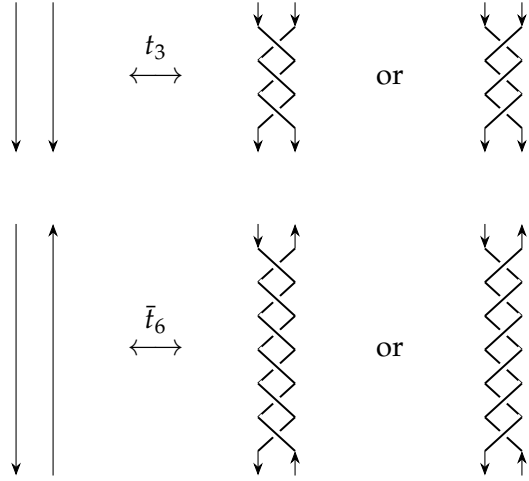
\begin{figure}[htbp]\centering
\begin{tikzpicture}
% t_3: two parallel strands <-> three half-twists (either hand)
\begin{scope}[scale=0.72]
\draw[-{Stealth[length=1.7mm]}] (-0.34,1.38)--(-0.34,-1.38);
\draw[-{Stealth[length=1.7mm]}] (0.34,1.38)--(0.34,-1.38);
\end{scope}
\node at (1.6,0) {$\longleftrightarrow$};\node at (1.6,0.4) {$t_3$};
\foreach \sx/\h in {3.2/1,6.2/-1}{
 \begin{scope}[shift={(\sx,0)},scale=0.72]
 \twistcolumn{\h}{3}
 \draw[-{Stealth[length=1.7mm]}] (-0.34,1.38)--(-0.34,0.93);
 \draw[-{Stealth[length=1.7mm]}] (-0.34,-0.93)--(-0.34,-1.38);
 \draw[-{Stealth[length=1.7mm]}] (0.34,1.38)--(0.34,0.93);
 \draw[-{Stealth[length=1.7mm]}] (0.34,-0.93)--(0.34,-1.38);
 \end{scope}}
\node at (4.7,0) {or};
% bar t_6: two antiparallel strands <-> six half-twists (either hand)
\begin{scope}[shift={(0,-3.6)},scale=0.72]
\draw[-{Stealth[length=1.7mm]}] (-0.34,2.31)--(-0.34,-2.31);
\draw[-{Stealth[length=1.7mm]}] (0.34,-2.31)--(0.34,2.31);
\end{scope}
\node at (1.6,-3.6) {$\longleftrightarrow$};\node at (1.6,-3.2) {$\bar t_6$};
\foreach \sx/\h in {3.2/1,6.2/-1}{
 \begin{scope}[shift={(\sx,-3.6)},scale=0.72]
 \twistcolumn{\h}{6}
 \draw[-{Stealth[length=1.7mm]}] (-0.34,2.31)--(-0.34,1.86);
 \draw[-{Stealth[length=1.7mm]}] (-0.34,-1.86)--(-0.34,-2.31);
 \draw[-{Stealth[length=1.7mm]}] (0.34,1.86)--(0.34,2.31);
 \draw[-{Stealth[length=1.7mm]}] (0.34,-2.31)--(0.34,-1.86);
 \end{scope}}
\node at (4.7,-3.6) {or};
\end{tikzpicture}
\caption{The oriented local replacements, with the exterior diagram fixed. }
\label{fig:moves}
\end{figure}

In 1987, Przytycki proposed the following conjecture
\cite{Kirby}.

\begin{conjecture}[Przytycki]
	\label{conj:t3-t6}
	Every oriented link can be transformed into an unlink by a finite
	sequence of \(t_3\)- and \(\bar t_6\)-moves and isotopies.
\end{conjecture}

In 2002, D\k{a}bkowski and Przytycki used Burnside groups of links to prove
that certain links cannot be transformed into any unlink by
\(3\)-moves, thereby disproving Conjecture~\ref{conj:3move}
\cite{DP}.

Their result also disproves Conjecture~\ref{conj:t3-t6}.
Indeed, after orientations are forgotten, a \(t_3\)-move is a single
\(3\)-move, while a \(\bar t_6\)-move can be decomposed into two
\(3\)-moves.
Thus, if an oriented link can be transformed into an unlink by
\(t_3\)- and \(\bar t_6\)-moves, then the underlying unoriented link
can also be transformed into an unlink by \(3\)-moves.
Consequently, the counterexamples constructed by D\k{a}bkowski and
Przytycki, with any choice of orientations, also provide counterexamples
to Conjecture~\ref{conj:t3-t6}.

Knot counterexamples also follow from the braid example in
\cite{DP}. Indeed, coherently oriented \(t_3\)-moves can join its
components to produce a knot without changing its equivalence
class under the oriented moves.\footnote{Let \(\gamma=(\sigma_1\sigma_2\sigma_3\sigma_4)^{10}\),
whose closure is a counterexample to the \(3\)-move conjecture
\cite{DP}. Orient all strands downwards and append
\(\sigma_1^3\sigma_2^3\sigma_3^3\sigma_4^3\).
These four \(t_3\)-moves change the closure permutation from the
identity to a five-cycle, so the new closure is a knot. If this knot
were reducible to an unlink by \(3\)-moves, then so would be
\(\widehat\gamma\), a contradiction.}
The knots obtained in this way are still not reducible to an unlink
by unoriented \(3\)-moves.  We consider a more specific question:
can a knot be transformed into an unlink by  \(3\)-moves,
yet fail to be transformed into any unlink by the oriented
\(t_3\)- and \(\bar t_6\)-moves?

Write \(L\sim L'\) when the oriented links \(L,L'\) are related by
 isotopies respecting their orientations and finitely many
\(t_3\)- and \(\bar t_6\)-moves.
Let \(U_m\) denote the \(m\)-component unlink, where \(m\geq1\).
The target is allowed to have any number of components, since these
local moves can change that number.
Our first result gives a concrete distinction between the two
notions of reduction.

\begin{thmA}\label{thm:A}
	The knot \(9_{46}\), with either orientation, is not
	\(\sim\)-equivalent to \(U_m\) for any \(m\geq1\).
	On the other hand, after its orientation is forgotten, it can be
	transformed into \(U_3\) by three \(3\)-moves.
	Both statements also hold for its mirror image.
\end{thmA}

The proof uses two finite Alexander quandles over the field \(\F\)
of three elements:
\[
R_3=\Aff(\F,-1),\qquad
J_9=\Aff\bigl(\R,\theta\bigr),\qquad
\R=\F[\eps]/(\eps^2),\quad \theta=-1+\eps.
\]
Here the operation on \(\Aff(A,t)\) is
\(x*y=tx+(1-t)y\).
Let \(N_3(L)\) and \(N_9(L)\) denote the corresponding numbers of
colorings, including constant colorings.
We will verify directly that \(N_3\) and \(N_9\) are invariant under
the oriented moves above and compute
\[
\bigl(N_3(9_{46}),N_9(9_{46})\bigr)=(27,81),
\qquad
\bigl(N_3(U_m),N_9(U_m)\bigr)=(3^m,9^m).
\]
The two counts cannot simultaneously agree with those of any unlink,
giving the oriented obstruction in Theorem~\ref{thm:A}.

The discrepancy between these coloring counts also reflects a
lifting problem.
The reduction map
\[
\R\longrightarrow\F,
\qquad
a+\eps b\longmapsto a
\]
sends every \(J_9\)-coloring to an \(R_3\)-coloring,
but a given \(R_3\)-coloring need not lift to a \(J_9\)-coloring.
To describe this obstruction, we define
\begin{equation}\label{eq:introdelta}
	\delta(L)=2\log_3 N_3(L)-\log_3 N_9(L),
\end{equation}
which we call the coloring lifting defect of \(L\).

The invariance of the coloring counts implies that \(\delta\)
is invariant under \(t_3\)- and \(\bar t_6\)-moves.
The computations above give
\[
\delta(U_m)=0,
\qquad
\delta(9_{46})=2.
\]
In Section~\ref{sec:lifting}, we will prove that \(\delta(L)\)
is a nonnegative integer equal to the rank of a linear lifting
obstruction map.
Moreover, the proportion of all \(R_3\)-colorings that lift to
\(J_9\)-colorings is exactly \(3^{-\delta(L)}\).
Thus, \(\delta(L)\) provides an obstruction to transformation into
an unlink and quantifies the obstruction to lifting \(3\)-colorings.

For knots, this defect can be computed from a Seifert matrix
and has the following properties.

\begin{thmA}\label{thm:B}
	Let \(V\) be a Seifert matrix of a knot \(K\).
	Reduce its entries modulo \(3\), retaining the notation \(V\), and set
	\[
	W=V+V^{\tr},\qquad E_K=\ker W,\qquad
	\beta_K(x,y)=x^{\tr}Vy\quad(x,y\in E_K).
	\]
	Then \(\beta_K\) is an alternating bilinear form, and
	\[
	\delta(K)=\rank\beta_K.
	\]
	In particular, \(\delta(K)\) is always a nonnegative even integer.
	For any oriented knots \(K_1,K_2\), we have
	\[
	\delta(K_1\#K_2)=\delta(K_1)+\delta(K_2).
	\]
\end{thmA}

Additivity allows us to pass from a single example to an infinite
family. Set
\[
K_s=\underbrace{9_{46}\#\cdots\#9_{46}}_{s\text{ copies}},
\qquad s\geq1.
\]

\begin{thmA}\label{thm:C}
	For every \(s\geq1\),
	\[
	N_3(K_s)=3^{2s+1},\qquad
	N_9(K_s)=3^{2s+2},\qquad
	\delta(K_s)=2s.
	\]
The knots \(K_s\) are pairwise inequivalent under \(\sim\), and no
\(K_s\) is \(\sim\)-equivalent to any unlink.
After forgetting orientation, \(K_s\) can be transformed into
\(U_{2s+1}\) by \(3s\) ordinary \(3\)-moves.
The same conclusions hold for the family of mirror images of the
\(K_s\).
\end{thmA}

Together with Theorem~\ref{thm:B} and \(\delta(U_1)=0\), this shows that
the set of values attained by \(\delta\) on knots is exactly
\(\{0,2,4,\ldots\}\).

Two further applications clarify what the defect detects.
Example~\ref{ex:same-alexander} compares \(9_{46}\) and \(6_1\):
they have the same Alexander polynomial and the same \(N_9\), but
different defects.
Using a Seifert-matrix result of Duzhin and Shkolnikov,
Corollary~\ref{cor:matched} proves that the defect vanishes for
knots admitting matched diagrams. Consequently, no \(K_s\) is
\(\sim\)-equivalent to any knot in that class.

The paper is organized as follows.
Section~\ref{sec:color} fixes the conventions for Alexander quandle
colorings and proves that both coloring counts are invariant under
\(t_3\)- and \(\bar t_6\)-moves.
Section~\ref{sec:knot} computes the coloring counts of \(9_{46}\)
and gives its reduction to an unlink by \(3\)-moves,
completing the proof of Theorem~\ref{thm:A}.
It also computes the lifting defect for the two link counterexamples
of D\k{a}bkowski and Przytycki.
Section~\ref{sec:lifting} develops the linear equations governing
coloring lifts and explains the nonnegativity of \(\delta\)
and its interpretation as a lifting obstruction.
Section~\ref{sec:seifert} gives the Seifert-form expression,
proves the connected-sum formula, and completes the proofs of
Theorems~\ref{thm:B} and~\ref{thm:C}, together with the two applications above.
\section{Colorings and the oriented local moves}\label{sec:color}\label{sec:quandles}
The obstruction in Theorem~\ref{thm:A} comes from two coloring counts.
We first fix the crossing convention and show that each count is preserved
by the oriented moves of Figure~\ref{fig:moves}.
The essential point is local: for either coloring target, a twist block
and the corresponding untwisted block admit exactly the same boundary
colorings, each with a unique extension to the interior.

\subsection{Quandle colorings}
\begin{definition}
A \emph{quandle} is a nonempty set $X$ with a binary operation $*$ such that,
for all $x,y,z\in X$,
\begin{enumerate}[label=\textup{(\arabic*)},nosep]
\item $x*x=x$;
\item the right translation $R_y\colon X\to X$, $x\mapsto x*y$, is bijective;
\item $(x*y)*z=(x*z)*(y*z)$.
\end{enumerate}
The \emph{inverse operation} is defined by $x*^{-1}y=R_y^{-1}(x)$. Thus
\[
(x*^{-1}y)*y=x=(x*y)*^{-1}y.
\]
\end{definition}
An $X$-coloring of an oriented link assigns an element of $X$ to each arc,
subject to the crossing rule in Figure~\ref{fig:conventions}.

\begin{figure}[htbp]\centering
\begin{minipage}{.48\textwidth}\centering
\begin{tikzpicture}[line width=.8pt,>=Stealth]
\draw (-1,0)--(-.12,0);\draw (.12,0)--(1,0);
\draw[->] (0,.75)--(0,-.75);
\node at (-.63,.24) {$a$};\node at (.67,.24) {$c$};
\node at (-.24,.54) {$b$};
\draw[->,dashed] (.16,-.38)--(.83,-.38);
\node at (0,-1.1) {$c=a*b$};
\end{tikzpicture}
\end{minipage}\hfill
\begin{minipage}{.48\textwidth}\centering
\begin{tikzpicture}[line width=.8pt,>=Stealth]
\draw (-1,0)--(-.12,0);\draw (.12,0)--(1,0);
\draw[->] (0,-.75)--(0,.75);
\node at (-.63,.24) {$a$};\node at (.67,.24) {$c$};
\node at (-.24,.54) {$b$};
\draw[->,dashed] (.83,-.38)--(.16,-.38);
\node at (0,-1.1) {$a=c*b\quad\Longleftrightarrow\quad c=a*^{-1}b$};
\end{tikzpicture}
\end{minipage}
\caption{Crossing rules for the two over-strand orientations.}
\label{fig:conventions}
\end{figure}
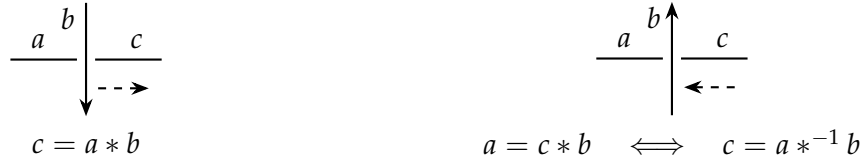

The three quandle axioms give bijections of coloring sets under the
oriented Reidemeister moves of types I, II and III, respectively
\cite{Joyce}. Hence the coloring set $\Col_X(L)$ is defined up to bijection
by the oriented link $L$. For finite $X$, its cardinality is therefore an
oriented link invariant. Throughout the paper, we count all colorings,
including constant colorings. In a crossingless diagram of $U_m$, the
$m$ components may be colored independently, so
\begin{equation}\label{eq:unlink}
 |\Col_X(U_m)|=|X|^m.
\end{equation}
For the coordinatewise quandle structure on $X\times Y$, a coloring is
precisely a pair of colorings by $X$ and $Y$. Consequently,
\[
|\Col_{X\times Y}(L)|=|\Col_X(L)|\,|\Col_Y(L)|.
\]

\subsection{Alexander quandles and the two targets}
Let $A$ be an abelian group and let $t$ be an automorphism of $A$.
The \emph{Alexander quandle} $\Aff(A,t)$ has underlying set $A$ and operations
\[
x*y=tx+(1-t)y,\qquad
x*^{-1}y=t^{-1}x+(1-t^{-1})y.
\]
The first quandle axiom follows by setting $y=x$, and the second follows
from the displayed inverse operation. The third follows from
\[
(x*y)*z=(x*z)*(y*z)=t^2x+t(1-t)y+(1-t)z.
\]
We will also use this construction when $A$ is a module over a commutative
ring and $t$ is a unit acting by scalar multiplication.

Our first target is the three-element quandle $R_3=\Aff(\F,-1)$, over the field
$\F$ of three elements. Here $x*y=-x+2y$ and $x*^{-1}y=x*y$. Its operation table,
with the first argument indexing the rows, is
\[
\begin{array}{c|ccc}
 *&0&1&2\\ \hline
 0&0&2&1\\
 1&2&1&0\\
 2&1&0&2
\end{array}.
\]
Thus the crossing relation $c=a*b$ is $a+c=2b$, the usual Fox
$3$-coloring rule.

The second target records additional information over the same field.
Set
\[
\R=\F[\eps]/(\eps^2),\qquad \theta=-1+\eps,
\qquad J_9=\Aff(\R,\theta).
\]
Every element of $\R$ has a unique expression $a+\eps b$, with $a,b\in\F$.
The ring has nine elements, but it is neither $\mathbb Z/9$ nor
$\mathbb F_9$: its additive group is $\F^2$ and $\eps$ is a nonzero
nilpotent. 
The operations on the ring are given by
\[
\begin{aligned}
	(a+\eps b)+(c+\eps d)&=(a+c)+\eps(b+d),\\
	(a+\eps b)(c+\eps d)&=ac+\eps(ad+bc).
\end{aligned}
\]
Since $\theta^{-1}=-1-\eps$, the element $\theta$ is a unit.
In \(J_9=\Aff(\R,\theta)\), the automorphism is \emph{multiplication by \(\theta\)}. Therefore
\[
(a+\eps b)*(c+\eps d)
=(-a+2c)+\eps(a-b-c+2d).
\]
In coordinates $(a,b)\leftrightarrow a+\eps b$, the two operations are
\[
\begin{aligned}
 (a,b)*(c,d)&=(-a+2c,\ a-b-c+2d),\\
 (a,b)*^{-1}(c,d)&=(-a+2c,\ -a-b+c+2d),
\end{aligned}
\]
with all entries taken modulo $3$.

The reduction map
\[
\rho\colon J_9\longrightarrow R_3,\qquad a+\eps b\longmapsto a
\]
is a quandle homomorphism. This map will give rise to the lifting problem
studied in Section~\ref{sec:lifting}. We write
\[
 N_3(L)=|\Col_{R_3}(L)|,\qquad N_9(L)=|\Col_{J_9}(L)|.
\]

\subsection{Crossing maps and twist blocks}
To compare the two sides of a local move, order the colors on each
horizontal level from left to right and propagate them geometrically
from top to bottom. This convention is independent of the strand
orientations. We call a crossing \emph{left-over} if its upper-left
endpoint lies on the over-strand, and \emph{right-over} otherwise.

The four cases are shown in Figure~\ref{fig:single-crossings}. Here $S^{-1}$ and $C^{-1}$ denote inverse maps on $X^2$, while $*^{-1}$
denotes the inverse quandle operation. For instance,
$(y*^{-1}x)*x=y$ gives $S^{-1}S(x,y)=(x,y)$.

\begin{figure}[htbp]\centering
\begin{minipage}[t]{.48\textwidth}\centering
\textbf{(a)} $S$\par\smallskip
\begin{tikzpicture}[line cap=round,font=\small]
\draw[line width=.7pt] (0.660,.90)--(-0.660,-.90);
\draw[white,line width=5pt] (-0.660,.90)--(0.660,-.90);
\draw[line width=1.05pt] (-0.660,.90)--(0.660,-.90);
\draw[-{Stealth[length=2.2mm,width=1.6mm]},line width=1.05pt]
 (-0.475,0.648)--(-0.198,0.270);
\node at (-.66,1.15) {$x$}; \node at (.66,1.15) {$y$};
\node at (-.66,-1.15) {$u$}; \node at (.66,-1.15) {$x$};
\end{tikzpicture}\par\small $y=u*x,\quad u=y*^{-1}x$\par\smallskip
\small $S(x,y)=(u,x)=(y*^{-1}x,x)$\end{minipage}
\hfill\begin{minipage}[t]{.48\textwidth}\centering
\textbf{(b)} $S^{-1}$\par\smallskip
\begin{tikzpicture}[line cap=round,font=\small]
\draw[line width=.7pt] (-0.660,.90)--(0.660,-.90);
\draw[white,line width=5pt] (0.660,.90)--(-0.660,-.90);
\draw[line width=1.05pt] (0.660,.90)--(-0.660,-.90);
\draw[-{Stealth[length=2.2mm,width=1.6mm]},line width=1.05pt]
 (0.475,0.648)--(0.198,0.270);
\node at (-.66,1.15) {$x$}; \node at (.66,1.15) {$y$};
\node at (-.66,-1.15) {$y$}; \node at (.66,-1.15) {$v$};
\end{tikzpicture}\par\small $v=x*y$\par\smallskip
\small $S^{-1}(x,y)=(y,v)=(y,x*y)$\end{minipage}
\par\medskip\begin{minipage}[t]{.48\textwidth}\centering
\textbf{(c)} $C$\par\smallskip
\begin{tikzpicture}[line cap=round,font=\small]
\draw[line width=.7pt] (0.660,.90)--(-0.660,-.90);
\draw[white,line width=5pt] (-0.660,.90)--(0.660,-.90);
\draw[line width=1.05pt] (-0.660,.90)--(0.660,-.90);
\draw[-{Stealth[length=2.2mm,width=1.6mm]},line width=1.05pt]
 (-0.198,0.270)--(-0.475,0.648);
\node at (-.66,1.15) {$x$}; \node at (.66,1.15) {$y$};
\node at (-.66,-1.15) {$u$}; \node at (.66,-1.15) {$x$};
\end{tikzpicture}\par\small $u=y*x$\par\smallskip
\small $C(x,y)=(u,x)=(y*x,x)$\end{minipage}
\hfill\begin{minipage}[t]{.48\textwidth}\centering
\textbf{(d)} $C^{-1}$\par\smallskip
\begin{tikzpicture}[line cap=round,font=\small]
\draw[line width=.7pt] (-0.660,.90)--(0.660,-.90);
\draw[white,line width=5pt] (0.660,.90)--(-0.660,-.90);
\draw[line width=1.05pt] (0.660,.90)--(-0.660,-.90);
\draw[-{Stealth[length=2.2mm,width=1.6mm]},line width=1.05pt]
 (0.198,0.270)--(0.475,0.648);
\node at (-.66,1.15) {$x$}; \node at (.66,1.15) {$y$};
\node at (-.66,-1.15) {$y$}; \node at (.66,-1.15) {$v$};
\end{tikzpicture}\par\small $x=v*y,\quad v=x*^{-1}y$\par\smallskip
\small $C^{-1}(x,y)=(y,v)=(y,x*^{-1}y)$\end{minipage}
\caption{The four single-crossing maps, read from top to bottom. }
\label{fig:single-crossings}
\end{figure}
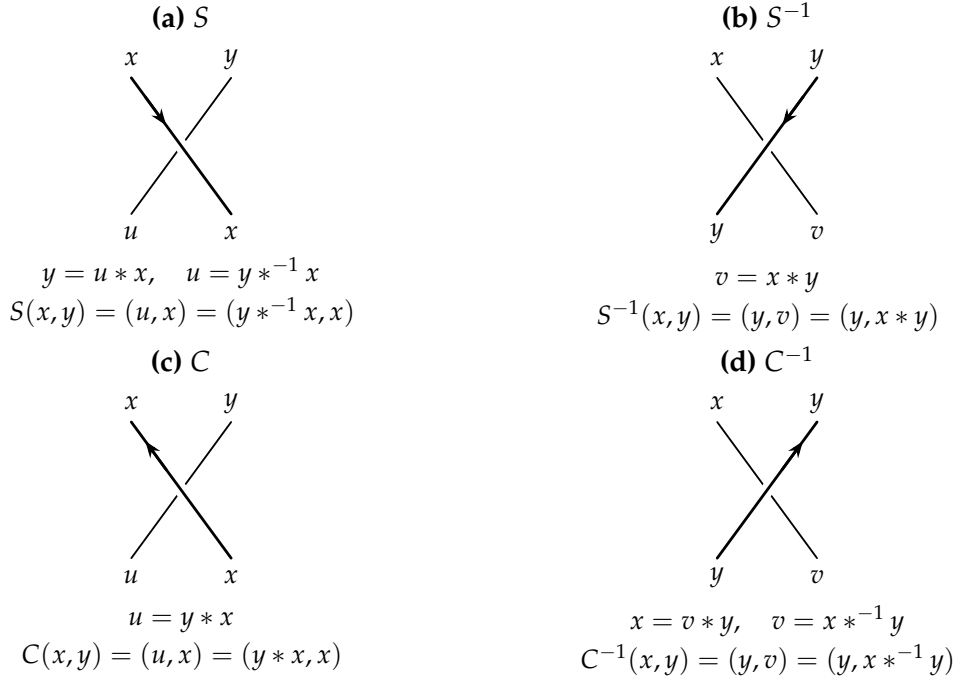

We compose maps with the rightmost factor acting first, so $CS=C\circ S$.
For the interchange map $J(x,y)=(y,x)$, the formulas above give
\[
C=JS^{-1}J,\qquad S=JC^{-1}J.
\]
These identities will allow us to treat the different orientation
choices together.

We now compose the crossing maps to obtain the boundary maps of the twist blocks.
For parallel strands pointing downwards, three left-over crossings give
$S^3$, and three right-over crossings give $S^{-3}$. When both strands
point upwards, the corresponding maps are $C^3$ and $C^{-3}$.
For antiparallel strands, the two positions are exchanged at every
crossing, so the direction of the over-strand alternates. With initial
directions $(\downarrow,\uparrow)$, six left-over crossings are read as
$S,C,S,C,S,C$, giving $(CS)^3$. Six right-over crossings are read as
$C^{-1},S^{-1},C^{-1},S^{-1},C^{-1},S^{-1}$, giving $(CS)^{-3}$.
Interchanging the initial directions interchanges $CS$ and $SC$.
Thus all the twist blocks in Figure~\ref{fig:moves} have the following
boundary maps:
\begin{equation}\label{eq:boundarytable}
\begin{array}{c|c|c|c}
\text{upper directions}&\text{half-twists}&\text{left-over}&\text{right-over}\\\hline
(\downarrow,\downarrow)&3&S^3&S^{-3}\\
(\uparrow,\uparrow)&3&C^3&C^{-3}\\
(\downarrow,\uparrow)&6&(CS)^3&(CS)^{-3}\\
(\uparrow,\downarrow)&6&(SC)^3&(SC)^{-3}
\end{array}.
\end{equation}
The corresponding untwisted block has boundary map $I$.

\begin{lemma}\label{lem:boundary}
Let $X$ be a quandle for which $S^3=I$ and $(CS)^3=I$ on $X^2$.
If two oriented diagrams $D$ and $D'$ differ by a $t_3$- or
$\bar t_6$-move inside a disk, their $X$-colorings are in bijection by
keeping all colors outside the disk fixed. For finite $X$, the coloring
count is therefore invariant under both moves.
\end{lemma}
\begin{proof}
The identity $C=JS^{-1}J$ implies $C^3=I$, and
\[
(SC)^3=S(CS)^3S^{-1}=I.
\]
Taking inverses shows that every map in \eqref{eq:boundarytable} is $I$.

Prescribe colors $(x,y)$ at the two upper endpoints of either tangle
in the disk. At each crossing, one of the maps $S^{\pm1},C^{\pm1}$
uniquely determines the lower pair from the upper pair. Thus the colors
extend uniquely through the block, and its boundary map determines the
two lower colors. Both the twisted and untwisted blocks have boundary
map $I$. Consequently, for any colors prescribed at all four endpoints,
either block admits exactly one extension when the lower pair equals
the upper pair, and no extension otherwise.

\end{proof}

\subsection{The identities for the two targets}
For an Alexander quandle $\Aff(A,t)$, recall that
\[
x*y=tx+(1-t)y,\qquad
x*^{-1}y=t^{-1}x+(1-t^{-1})y.
\]
Substituting these operations into the four crossing maps of
Figure~\ref{fig:single-crossings}(a)--(d), in that order, gives
\begin{align*}
S(x,y)&=(y*^{-1}x,x)=\bigl((1-t^{-1})x+t^{-1}y,\ x\bigr),\\
S^{-1}(x,y)&=(y,x*y)=\bigl(y,\ tx+(1-t)y\bigr),\\
C(x,y)&=(y*x,x)=\bigl((1-t)x+ty,\ x\bigr),\\
C^{-1}(x,y)&=(y,x*^{-1}y)=\bigl(y,\ t^{-1}x+(1-t^{-1})y\bigr).
\end{align*}
We represent both the top and bottom color pairs by column vectors,
ordered from left to right. The first row of a crossing matrix
therefore consists of the coefficients of $x,y$ in the lower-left
color, and the second row gives those in the lower-right color.
For example, the formula for $S$ gives the first row
$(1-t^{-1},t^{-1})$ and the second row $(1,0)$.
Reading off the coefficients in all four formulas yields
\begin{equation}\label{eq:SC}
\begin{aligned}
 [S]&=\begin{pmatrix}1-t^{-1}&t^{-1}\\1&0\end{pmatrix}, &
 [C]&=\begin{pmatrix}1-t&t\\1&0\end{pmatrix},\\[2pt]
 [S]^{-1}&=\begin{pmatrix}0&1\\t&1-t\end{pmatrix}, &
 [C]^{-1}&=\begin{pmatrix}0&1\\t^{-1}&1-t^{-1}\end{pmatrix}.
\end{aligned}
\end{equation}
With this column-vector convention, applying $S$ first and then $C$
corresponds to multiplying by $[C][S]$, so $[C\circ S]=[C][S]$. Set
\[
u=t^{-1},\qquad
H=\begin{pmatrix}1&-1\\1&-1\end{pmatrix}.
\]
Direct multiplication gives
\begin{align}
 [S]^3-I&=(t^{-2}-t^{-1}+1)
 \begin{pmatrix}-t^{-1}&t^{-1}\\1&-1\end{pmatrix},\label{eq:S3}\\
 [D]:=[C][S]&=\begin{pmatrix}2-u&u-1\\1-u&u\end{pmatrix}
 =I+(1-u)H.\label{eq:D}
\end{align}
Since $H^2=0$, expanding the cube of $I+(1-u)H$ yields
\begin{equation}\label{eq:D3}
 [D]^3-I=3(1-u)H.
\end{equation}

\begin{lemma}\label{lem:criterion}
The crossing maps of an Alexander quandle $X=\Aff(A,t)$ satisfy
\[
\begin{aligned}
S^3=I&\quad\Longleftrightarrow\quad(t^2-t+1)A=0,\\
(CS)^3=I&\quad\Longleftrightarrow\quad3(t-1)A=0.
\end{aligned}
\]
\end{lemma}
\begin{proof}
By \eqref{eq:S3}, all entries of $[S]^3-I$ are multiples of
$t^{-2}-t^{-1}+1$.
Thus $S^3=I$ on $A^2$ if and only if $(t^{-2}-t^{-1}+1)A=0$.
Multiplying by the automorphism $t^2$ gives the first equivalence.

By \eqref{eq:D3}, the entries of $[D]^3-I$ are
$\pm3(1-t^{-1})$. Hence $(CS)^3=I$ if and only if
$3(1-t^{-1})A=0$, which is equivalent to $3(t-1)A=0$
since $t$ is invertible on $A$.
\end{proof}

\begin{proposition}\label{prop:invariance}
The counts $N_3$ and $N_9$ are invariant under $t_3$- and $\bar t_6$-moves,
for both handedness choices and all compatible strand orientations.
\end{proposition}
\begin{proof}
For $R_3$, the parameter is $t=-1$, so $t^2-t+1=3=0$ and
$3(t-1)=-6=0$. For $J_9$, the parameter is $t=\theta$, and both
$\theta^2-\theta+1$ and $3$ vanish in $\R$.
Lemma~\ref{lem:criterion} therefore gives $S^3=I$ and $(CS)^3=I$
for both targets. The conclusion follows from Lemma~\ref{lem:boundary}.
\end{proof}

We will also need to compare the two orientations of a knot and its
mirror image.
\begin{lemma}\label{lem:mirror}
The counts $N_3$ and $N_9$ are unchanged by taking a mirror image or by
reversing the orientations of all components simultaneously.
\end{lemma}
\begin{proof}
For an oriented link $L$, write $mL$ for its mirror image, $rL$ for
the result of reversing all component orientations, and
$c_X(L)=|\Col_X(L)|$. For a finite quandle $X$ with dual
$X^*=(X,*^{-1})$, the standard duality relation is
\[
c_X(rL)=c_X(mL)=c_{X^*}(L).
\]
For knots, this is Lemma~3.2 of~\cite{CESY}; the same local coloring relations
give this identity for links. It therefore suffices to show that
$R_3$ and $J_9$ are each isomorphic to their duals.

For $R_3$, the operations $*$ and $*^{-1}$ coincide, so the identity
map is such an isomorphism. For $J_9$, the ring involution
\[
\alpha(a+\eps b)=a-\eps b
\]
satisfies $\alpha(\theta)=\theta^{-1}$ and hence
\[
\alpha(x*y)
=\theta^{-1}\alpha(x)+(1-\theta^{-1})\alpha(y)
=\alpha(x)*^{-1}\alpha(y).
\]
Thus $\alpha$ is an isomorphism from $J_9$ to its dual. Applying the
duality relation to these two targets proves both assertions.
\end{proof}

\section{The knot \texorpdfstring{$9_{46}$}{9(46)} and two link examples}\label{sec:knot}\label{sec:example}
We prove Theorem~\ref{thm:A} by computing the two coloring counts of
$9_{46}$ and reducing its unoriented pretzel diagram to $U_3$.

\subsection{The two coloring counts}
We use the four-strand braid
\begin{equation}\label{eq:braid}
\beta=\sigma_1^{-1}\sigma_2\sigma_1^{-1}\sigma_2\sigma_3^{-1}
\sigma_2^{-1}\sigma_1\sigma_2^{-1}\sigma_3^{-1},
\end{equation}
whose closure represents $9_{46}$ up to mirror image \cite{Atlas}.
All four strands are oriented downwards. We take $\sigma_i$ to be the
left-over crossing in Figure~\ref{fig:single-crossings}(a), acting on
positions $i,i+1$, and read the word from top to bottom.
Different braid conventions may give the mirror image.
By Lemma~\ref{lem:mirror}, the coloring counts are the same.
The pretzel representation $P(3,3,-3)$ of the
same knot is also recorded in \cite{DS}.

\begin{proposition}\label{prop:knotcounts}
For either orientation of $9_{46}$ and of its mirror image,
\begin{equation}\label{eq:knotcounts}
N_3(9_{46})=27,\qquad N_9(9_{46})=81,\qquad\delta(9_{46})=2.
\end{equation}
\end{proposition}
\begin{proof}
By the crossing formulas, $\sigma_i$ sends $(x,y)$ to $(y*^{-1}x,x)$,
whereas $\sigma_i^{-1}$ sends $(x,y)$ to $(y,x*y)$.
Each top tuple extends uniquely through the braid. Closing the braid
identifies each lower endpoint with the corresponding upper endpoint,
so
\[
\Col_X(\widehat\beta)
\cong\{x\in X^4:\mathcal B_{\beta,X}(x)=x\},
\]
where $\mathcal B_{\beta,X}$ denotes the composite crossing map.

To compute this map, put
\[
f(x,y)=tx+(1-t)y,\qquad
 g(x,y)=t^{-1}x+(1-t^{-1})y.
\]
The nine successive updates are as follows:
\begin{center}\small
\begin{tabular}{c|c|l|l}
step&letter&new color&tuple below\\ \hline
0&---&---&$(x_1,x_2,x_3,x_4)$\\
1&$\sigma_1^{-1}$&$u_1=f(x_1,x_2)$&$(x_2,u_1,x_3,x_4)$\\
2&$\sigma_2$&$u_2=g(x_3,u_1)$&$(x_2,u_2,u_1,x_4)$\\
3&$\sigma_1^{-1}$&$u_3=f(x_2,u_2)$&$(u_2,u_3,u_1,x_4)$\\
4&$\sigma_2$&$u_4=g(u_1,u_3)$&$(u_2,u_4,u_3,x_4)$\\
5&$\sigma_3^{-1}$&$u_5=f(u_3,x_4)$&$(u_2,u_4,x_4,u_5)$\\
6&$\sigma_2^{-1}$&$u_6=f(u_4,x_4)$&$(u_2,x_4,u_6,u_5)$\\
7&$\sigma_1$&$u_7=g(x_4,u_2)$&$(u_7,u_2,u_6,u_5)$\\
8&$\sigma_2^{-1}$&$u_8=f(u_2,u_6)$&$(u_7,u_6,u_8,u_5)$\\
9&$\sigma_3^{-1}$&$u_9=f(u_8,u_5)$&$(u_7,u_6,u_5,u_9)$
\end{tabular}
\end{center}
The closure equations are therefore
\[
u_7=x_1,\qquad u_6=x_2,\qquad u_5=x_3,\qquad u_9=x_4.
\]
The braid $\beta$ has four strands, so its colors form a column
$(x_1,x_2,x_3,x_4)^{\tr}$. We use the $2\times2$ crossing matrix $[S]$
from \eqref{eq:SC} to define three $4\times4$ matrices:
\[
S_1=[S]\oplus I_2,\qquad
S_2=(1)\oplus[S]\oplus(1),\qquad
S_3=I_2\oplus[S].
\]
Here $\oplus$ denotes block diagonal sum, and $I_2$ is the $2\times2$
identity matrix. Thus $S_i$ acts on entries $i,i+1$ and leaves the
other two entries unchanged.
The matrix for $\sigma_i^{-1}$ is $S_i^{-1}$.
Applying the nine crossings in order gives
\[
B=S_3^{-1}S_2^{-1}S_1S_2^{-1}S_3^{-1}S_2S_1^{-1}S_2S_1^{-1},
\]
where the first crossing gives the rightmost factor because the colors
are column vectors. Substitution of $t=\theta$ gives
$B-I_4=P_0+\eps P_1$, with
\[
 P_0=\begin{pmatrix}1&2&1&2\\1&2&1&2\\1&2&1&2\\1&2&1&2\end{pmatrix},\qquad
 P_1=\begin{pmatrix}0&1&0&2\\1&0&0&2\\1&1&2&2\\1&1&0&1\end{pmatrix}.
\]

At $t=-1$, the fixed-point equation is simply $P_0v=0$ over $\F$.
Since all rows of $P_0$ are equal and nonzero, it reduces to
\[
v_1-v_2+v_3-v_4=0.
\]
There are three free coordinates, so $N_3(9_{46})=3^3=27$.

For $J_9$, write a top tuple uniquely as $x=v+\eps a$, with
$v,a\in\F^4$. Equating the two coefficients in $(B-I_4)x=0$ gives
\[
P_0v=0,\qquad P_0a+P_1v=0.
\]
Set
\[
s=a_1+2a_2+a_3+2a_4\in\F.
\]
Since the four rows of $P_0$ are identical,
\[
P_0a=(s,s,s,s)^{\tr}.
\]
Using the displayed matrix $P_1$, the equation $P_0a+P_1v=0$ becomes
\[
\begin{cases}
s+v_2+2v_4=0,\\
s+v_1+2v_4=0,\\
s+v_1+v_2+2v_3+2v_4=0,\\
s+v_1+v_2+v_4=0.
\end{cases}
\]
Over $\F$, $2=-1$ and $-2=1$. Subtracting consecutive equations gives
\[
\begin{aligned}
\text{second minus first:}\quad
&v_1-v_2=0 \quad\Longrightarrow\quad v_1=v_2,\\
\text{third minus second:}\quad
&v_2+2v_3=0 \quad\Longrightarrow\quad v_2=v_3,\\
\text{fourth minus third:}\quad
&-2v_3-v_4=0 \quad\Longrightarrow\quad v_3=v_4.
\end{aligned}
\]
Therefore
\[
v_1=v_2=v_3=v_4.
\]
Let $c\in\F$ denote their common value.
For $v=c(1,1,1,1)^{\tr}$, direct substitution gives $P_0v=P_1v=0$.
The equation $P_0a+P_1v=0$ therefore reduces to $P_0a=0$. Each row gives
\[
a_1+2a_2+a_3+2a_4=0,
\qquad\text{or}\qquad a_1=a_2-a_3+a_4.
\]
The parameters $c,a_2,a_3,a_4\in\F$ are independent, so
\[
N_9(9_{46})=3^4=81.
\]
Thus
\[
\delta(9_{46})=2\log_3 N_3(9_{46})-\log_3 N_9(9_{46})
=2\cdot3-4=2.
\]
By Lemma~\ref{lem:mirror}, the same values hold for both orientations
of $9_{46}$ and its mirror image.
\end{proof}

\subsection{Oriented obstruction and unoriented reduction}
The two counts must be used together: either count alone agrees with
that of an unlink, but they do not agree with the same unlink.

\begin{corollary}\label{cor:separation}
The knot $9_{46}$ cannot be transformed into any unlink by oriented
$t_3$- and $\bar t_6$-moves. After its orientation is forgotten,
three $3$-moves transform it into $U_3$.
\end{corollary}
\begin{proof}[Proof of Corollary~\ref{cor:separation} and Theorem~\ref{thm:A}]
If $9_{46}\sim U_m$, Proposition~\ref{prop:invariance} and
\eqref{eq:unlink} would give
\[
27=3^m,\qquad81=9^m.
\]
The first equality forces $m=3$, while the second then gives the
contradiction $81=9^3=729$.

For the unoriented reduction, use the standard pretzel diagram
$P(3,3,-3)$; see \cite{DS} and Figure~\ref{fig:pretzel}.
After forgetting orientation, apply one ordinary $3$-move in each of
the three twist boxes. This removes all crossings and leaves three
disjoint circles, representing $U_3$.
By Lemma~\ref{lem:mirror}, the same coloring counts give the obstruction
for either orientation and for the mirror image.
The ordinary $3$-moves do not depend on orientation, and reflecting
them gives the reduction of the mirror knot to $U_3$.
\end{proof}

\begin{figure}[htbp]
\centering
\begin{tikzpicture}[x=1.15cm,y=0.85cm,line width=0.8pt,line cap=round]
  % Left: P(3,3,-3) with antiparallel top directions.
  \begin{scope}
    \foreach \cx/\hand in {0/1,1.2/1,2.4/-1}{
      \foreach \ytop in {2.4,1.6,0.8}{
        \ifnum\hand=1
          \draw (\cx+0.24,\ytop)--(\cx-0.24,\ytop-0.8);
          \draw[preaction={draw=white,line width=4pt}]
            (\cx-0.24,\ytop)--(\cx+0.24,\ytop-0.8);
        \else
          \draw (\cx-0.24,\ytop)--(\cx+0.24,\ytop-0.8);
          \draw[preaction={draw=white,line width=4pt}]
            (\cx+0.24,\ytop)--(\cx-0.24,\ytop-0.8);
        \fi
      }
      \draw (\cx-0.24,2.4)--(\cx-0.24,2.75);\draw (\cx+0.24,2.4)--(\cx+0.24,2.75);
      \draw (\cx-0.24,0)--(\cx-0.24,-0.35);\draw (\cx+0.24,0)--(\cx+0.24,-0.35);
      \draw[-{Stealth[length=1.8mm]}] (\cx-0.24,2.75)--(\cx-0.24,2.5);
      \draw[-{Stealth[length=1.8mm]}] (\cx+0.24,2.5)--(\cx+0.24,2.75);
      \draw[densely dashed,gray,line width=0.35pt]
        (\cx-0.37,0.05) rectangle (\cx+0.37,2.35);
    }
    \foreach \cx in {0,1.2}{
      \draw (\cx+0.24,2.75) .. controls (\cx+0.24,3.15) and
        (\cx+0.96,3.15) .. (\cx+0.96,2.75);
      \draw (\cx+0.24,-0.35) .. controls (\cx+0.24,-0.75) and
        (\cx+0.96,-0.75) .. (\cx+0.96,-0.35);
    }
    \draw (-0.24,2.75) .. controls (-0.24,4.7) and
      (2.64,4.7) .. (2.64,2.75);
    \draw (-0.24,-0.35) .. controls (-0.24,-2.3) and
      (2.64,-2.3) .. (2.64,-0.35);
    \node[align=center] at (1.2,-2.85)
      {$P(3,3,-3)$\\[-1pt]{\footnotesize up to mirror image}};
  \end{scope}
  \draw[-{Stealth[length=2mm]}] (3.05,1.2)--(4.85,1.2);
  \node[align=center,font=\footnotesize,anchor=south] at (3.95,1.45)
    {one ordinary\\$3$-move in\\each box};
  % Right: the crossingless diagram P(0,0,0) = U_3.
  \begin{scope}[xshift=6cm]
    \foreach \cx in {0,1.2,2.4}{
      \draw (\cx-0.24,2.75)--(\cx-0.24,-0.35);
      \draw (\cx+0.24,2.75)--(\cx+0.24,-0.35);
    }
    \foreach \cx in {0,1.2}{
      \draw (\cx+0.24,2.75) .. controls (\cx+0.24,3.15) and
        (\cx+0.96,3.15) .. (\cx+0.96,2.75);
      \draw (\cx+0.24,-0.35) .. controls (\cx+0.24,-0.75) and
        (\cx+0.96,-0.75) .. (\cx+0.96,-0.35);
    }
    \draw (-0.24,2.75) .. controls (-0.24,4.7) and
      (2.64,4.7) .. (2.64,2.75);
    \draw (-0.24,-0.35) .. controls (-0.24,-2.3) and
      (2.64,-2.3) .. (2.64,-0.35);
    \node at (1.2,-2.85) {$P(0,0,0)=U_3$};
  \end{scope}
\end{tikzpicture}
\caption{Left: the pretzel diagram $P(3,3,-3)$. Right:
an ordinary $3$-move in each box removes the three twist boxes and leaves three circles.}
\label{fig:pretzel}
\end{figure}
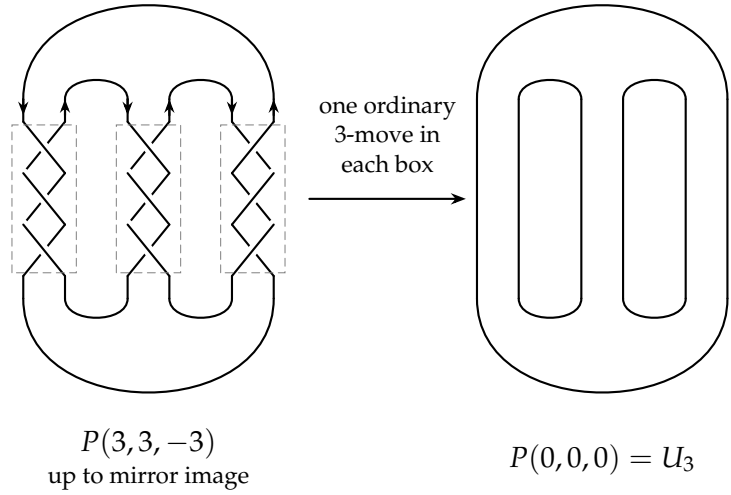

\begin{remark}\label{rem:burnside}
By Proposition~4 of~\cite{DP}, the third Burnside group $B_L(3)$
is invariant under $3$-moves. The unoriented reduction above therefore gives
\[
B_{9_{46}}(3)\cong B_{U_3}(3).
\]
By contrast, $N_9(9_{46})=81$ and $N_9(U_3)=729$.
Thus $N_9$, and likewise $\delta$, is not an invariant of unoriented
$3$-moves. The example separates the obstruction for the oriented moves
from the Burnside obstruction for unoriented $3$-moves.
\end{remark}

\subsection{Two link counterexamples}\label{subsec:DP-links}
By Theorem~6 of D\k{a}bkowski and Przytycki~\cite{DP}, the links
$\widehat\gamma$ and $L_{2BR}$ cannot be transformed into any unlink
by ordinary $3$-moves. We now compute their coloring counts
and lifting defects.

The first link $\widehat\gamma$ is the closure of the five-strand braid
\begin{equation}\label{eq:DP-gamma}
\gamma=(\sigma_1\sigma_2\sigma_3\sigma_4)^{10}.
\end{equation}
The second link is $L_{2BR}=\widehat{\beta}_{2BR}$, where
\begin{equation}\label{eq:DP-cable}
\beta_{2BR}=
(\sigma_2\sigma_1\sigma_3\sigma_2
 \sigma_4^{-1}\sigma_3^{-1}\sigma_5^{-1}\sigma_4^{-1})^3.
\end{equation}
This six-strand braid doubles each strand of the Borromean braid
$(\sigma_1\sigma_2^{-1})^3$; see Example~6.2 of~\cite{FK}.
The two four-letter blocks replace the positive and negative crossings,
respectively, by their blackboard two-parallels.
Every crossing of the original closed braid is between distinct
components, so each component has self-writhe zero.
Its blackboard parallel therefore has zero framing, and the closure
of $\beta_{2BR}$ is the zero-framed two-parallel of the Borromean rings.
The two links have five and six components, respectively.
All braid strands are oriented downwards.

\begin{proposition}\label{prop:DP-counts}
For these downward braid orientations, the two links have
\[
\begin{array}{c|c|c|c}
L&N_3(L)&N_9(L)&\delta(L)\\ \hline
\widehat\gamma&3^5&3^6&4\\
L_{2BR}&3^6&3^{10}&2
\end{array}.
\]
Neither oriented link can be transformed into any unlink by
$t_3$- and $\bar t_6$-moves.
\end{proposition}

\begin{proof}
Use the same matrix calculation as in Proposition~\ref{prop:knotcounts},
with five and six strands. It gives
\[
B_\gamma=I_5+\eps M_\gamma,\qquad
B_{2BR}=I_6+\eps M_{2BR},
\]
where
\begin{equation}\label{eq:DP-gamma-matrix}
M_\gamma=I_5-uw^{\tr},\qquad
u=(1,1,1,1,1)^{\tr},\quad w=(1,-1,1,-1,1)^{\tr},
\end{equation}
and
\begin{equation}\label{eq:DP-cable-matrix}
M_{2BR}=
\begin{pmatrix}
0&0&2&1&1&2\\
0&0&2&1&1&2\\
1&2&0&0&2&1\\
1&2&0&0&2&1\\
2&1&1&2&0&0\\
2&1&1&2&0&0
\end{pmatrix}.
\end{equation}
The kernel of $M_\gamma$ consists of constant columns.
The kernel of $M_{2BR}$ is given by
$v_1-v_2=v_3-v_4=v_5-v_6$.
Thus their ranks over $\F$ are $4$ and $2$, respectively.

For $B=I_n+\eps M$, write each top color column uniquely as
$x=v+\eps a$, where $v,a\in\F^n$. Since $\eps^2=0$,
\[
Bx=x+\eps Mv,\qquad Bx=x\ \Longleftrightarrow\ Mv=0.
\]
If $r=\rank M$, then $\ker M$ has dimension $n-r$.
Thus there are $3^{n-r}$ choices for $v$ and $3^n$ choices for $a$.
Reducing $B$ modulo $\eps$ gives $I_n$, so every column in $\F^n$
defines an $R_3$-coloring of the closure. Hence
\begin{equation}\label{eq:DP-rank-counts}
N_3=3^n,\qquad N_9=3^{n-r}\cdot3^n=3^{2n-r},\qquad\delta=r.
\end{equation}
Taking $(n,r)=(5,4)$ and $(6,2)$ gives the table.
The counts for $\gamma$ also follow from the Alexander--Burau matrix
for $\gamma^{-1}$ computed by Przytycki and Tsukamoto \cite{PT}.
Since every unlink has $\delta=0$, invariance of $\delta$ under
$t_3$- and $\bar t_6$-moves proves the last assertion.
\end{proof}

\begin{remark}\label{rem:DP-orientations}
The orientations in Proposition~\ref{prop:DP-counts} matter.
Number the components by their top endpoints from left to right.
Write $+$ for downward orientation and $-$ for upward orientation.
For the five-strand braid $\gamma$ in \eqref{eq:DP-gamma},
the orientation $(-,+,+,+,+)$ gives $B_\gamma=I_5$.
For the six-strand braid $\beta_{2BR}$ in \eqref{eq:DP-cable},
the orientation $(+,-,+,-,+,-)$ gives $B_{2BR}=I_6$.
Thus, for these orientations,
\[
\begin{aligned}
N_3(\widehat\gamma)&=3^5, & N_9(\widehat\gamma)&=3^{10},\\
N_3(L_{2BR})&=3^6, & N_9(L_{2BR})&=3^{12}.
\end{aligned}
\]
Both lifting defects are zero.
Our two counts do not exclude unlinks in these cases.
The Burnside obstruction in~\cite{DP} still excludes every unlink.
\end{remark}

\section{The lifting defect}\label{sec:lifting}\label{sec:defect}
Recall that reduction $\rho(a+\eps b)=a$ sends each $J_9$-coloring
to an $R_3$-coloring. We determine which $R_3$-colorings lift and how
many lifts they have. This gives a rank formula for $\delta(L)$.
All vector spaces, dimensions, and ranks in this section are over
$\F$.

\subsection{The equations for a lift}
Let $D$ be an oriented diagram of a link $L$, with $q$ arcs and $p$
crossings.
Number the arcs and write $x_1,\ldots,x_q$ for their colors.
At a crossing, let $j$ label the over-arc and choose the under-arc
labels $i,k$ according to Figure~\ref{fig:conventions}: take
$(x_i,x_j,x_k)=(a,b,c)$ in the left panel and
$(x_i,x_j,x_k)=(c,b,a)$ in the right panel.
In both cases, the coloring relation is
\begin{equation}\label{eq:crossing-row}
 x_k=x_i*x_j=t x_i+(1-t)x_j,
 \qquad\text{equivalently}\qquad
 x_k-tx_i-(1-t)x_j=0.
\end{equation}

Write the crossing equations as a linear system in $x_1,\ldots,x_q$.
For example, if $q=4$, a crossing with $(i,j,k)=(1,2,3)$ gives
\[
 -tx_1+(t-1)x_2+x_3=0.
\]
Its coefficients, in the order $x_1,x_2,x_3,x_4$, form the row
$(-t,t-1,1,0)$. The last entry is zero because $x_4$ does not occur.
Each crossing gives one row in this way.
If the same variable occurs more than once, combine its coefficients.
Putting these $p$ rows together gives the $p\times q$ matrix
\[
 F(t)\in\operatorname{Mat}_{p\times q}(\mathbb Z[t,t^{-1}]).
\]
For an Alexander quandle $\Aff(M,t)$, with $M$ a
$\mathbb Z[t,t^{-1}]$-module, a color vector $x\in M^q$ is a coloring
exactly when $F(t)x=0$. The matrix need not be square.
Each row sums to zero, so every constant color vector is a solution.

Reduce the coefficients of $F(t)$ modulo $3$, and keep the notation $F(t)$.
Set
\begin{equation}\label{eq:diagramF}
 F_0=F(-1),\qquad
 F_1=F'(-1)
 \quad\text{in }\operatorname{Mat}_{p\times q}(\F),
\end{equation}
where $F'(t)$ is obtained by differentiating each entry of $F(t)$.
Each entry has the form $f(t)=\alpha+\beta t$, with
$\alpha,\beta\in\F$. Substituting $\theta=-1+\eps$ gives
\[
 f(\theta)=\alpha+\beta(-1+\eps)
          =(\alpha-\beta)+\eps\beta
          =f(-1)+\eps f'(-1).
\]
Applying this to each entry gives
\[
 F(\theta)=F_0+\eps F_1.
\]

Write a color vector in $\R^q$ uniquely as $x=v+\eps a$, with
$v,a\in\F^q$. Since $\eps^2=0$,
\[
 F(\theta)(v+\eps a)=F_0v+\eps(F_0a+F_1v).
\]
For $x$ to be a $J_9$-coloring, the constant and $\eps$ coefficients
must both vanish. Hence
\begin{equation}\label{eq:two-lift-equations}
 F(\theta)(v+\eps a)=0
 \quad\Longleftrightarrow\quad
 F_0v=0\ \text{ and }\ F_0a=-F_1v.
\end{equation}

A vector $v\in\F^q$ is an $R_3$-coloring exactly when $F_0v=0$. Thus
\[
 E_D=\ker F_0=\Col_{R_3}(D),\qquad c=\dim E_D.
\]
Fix such a coloring $v$. For $a\in\F^q$, the color vector
$v+\eps a$ reduces to $v$ and is a \emph{lift} when it is a
$J_9$-coloring. By \eqref{eq:two-lift-equations}, finding a lift means
solving $F_0a=-F_1v$ for the unknown vector $a$.
Since $\im F_0=\{F_0a:a\in\F^q\}$ is a vector subspace, a solution
exists exactly when $F_1v\in\im F_0$. Define the linear map
\begin{equation}\label{eq:bD}
 b_D\colon E_D\longrightarrow\coker F_0
       =\F^p/\im F_0,
 \qquad v\longmapsto[F_1v].
\end{equation}
Here $[F_1v]$ denotes the class modulo $\im F_0$.
Its vanishing is precisely the solvability condition above, so
$\ker b_D$ is the subspace of liftable $R_3$-colorings.
We call $b_D$ the lifting obstruction map.

\subsection{Counting the lifts}
\begin{proposition}\label{prop:defect}
Let $D$ be a diagram of an oriented link $L$. Set
$c=\dim E_D$ and $r=\rank b_D$. Then
\begin{equation}\label{eq:count-rank}
 N_3(L)=3^c,\qquad N_9(L)=3^{2c-r},\qquad \delta(L)=r.
\end{equation}
In particular, $\delta(L)$ is a nonnegative integer and is invariant
under $t_3$- and $\bar t_6$-moves. Every unlink has $\delta=0$.
We also have
\begin{equation}\label{eq:block-rank}
 \delta(L)=
 \rank\begin{pmatrix}F_0&0\\F_1&F_0\end{pmatrix}
 -2\rank F_0.
\end{equation}
\end{proposition}
\begin{proof}
The $R_3$-colorings are exactly the vectors in $E_D=\ker F_0$.
This space has dimension $c$ over $\F$, so it contains $3^c$ vectors.
Hence $N_3(L)=3^c$.
As shown above, a coloring $v$ admits a lift exactly when $b_D(v)=0$.
Thus the liftable colorings are the vectors in $\ker b_D$.
By the rank--nullity theorem, this space has dimension $c-r$,
so there are $3^{c-r}$ choices for $v$.

Fix one such $v$. The equation $F_0a=-F_1v$ has a solution; choose
one and call it $a_0$. Since $F_0a_0=-F_1v$,
\[
 F_0a=-F_1v
 \quad\Longleftrightarrow\quad F_0a=F_0a_0
 \quad\Longleftrightarrow\quad F_0(a-a_0)=0.
\]
The last condition says that $a-a_0\in\ker F_0=E_D$.
Thus all solutions are exactly
\[
a=a_0+h,\qquad h\in E_D.
\]
Each of the $3^c$ choices for $h$ gives a lift $v+\eps(a_0+h)$.
The decomposition $x=v+\eps a$ is unique, so different choices of $h$
give different lifts. Hence this fixed $v$ has exactly $3^c$ lifts.
Every $J_9$-coloring reduces to a unique $v\in\ker b_D$.
Thus lifts of different $R_3$-colorings cannot coincide. Therefore
\[
 N_9(L)=3^{c-r}\,3^c=3^{2c-r}.
\]
Substituting these two counts into the definition of $\delta$ gives
\[
\begin{aligned}
\delta(L)&=2\log_3N_3(L)-\log_3N_9(L)\\
&=2c-(2c-r)=r.
\end{aligned}
\]
Since $N_3$ and $N_9$ are link invariants, $r=\delta(L)$ is independent
of $D$. Proposition~\ref{prop:invariance} shows that both counts are
unchanged by $t_3$- and $\bar t_6$-moves, so $\delta$ is also unchanged.
For a crossingless diagram of $U_m$, we have $p=0$, $q=m$, and
$E_D=\F^m$. The codomain of $b_D$ is zero, so $r=0$ and $\delta(U_m)=0$.

Finally, the $2p\times2q$ block matrix in \eqref{eq:block-rank} satisfies
\[
\begin{pmatrix}F_0&0\\F_1&F_0\end{pmatrix}
\begin{pmatrix}v\\a\end{pmatrix}
=\begin{pmatrix}F_0v\\F_1v+F_0a\end{pmatrix}.
\]
Its kernel consists of the solutions of \eqref{eq:two-lift-equations}.
The count above gives $3^{2c-r}$ solutions, so this kernel has
dimension $2c-r$. Since $c=q-\rank F_0$,
\[
 \rank\begin{pmatrix}F_0&0\\F_1&F_0\end{pmatrix}
 =2q-(2c-r)=2\rank F_0+r,
\]
which proves \eqref{eq:block-rank}.
\end{proof}

The liftable colorings form a subspace of $E_D$ of codimension
$\delta(L)$:
\[
 \delta(L)=\dim E_D-\dim\ker b_D,
 \qquad
 \frac{|\ker b_D|}{|E_D|}=3^{-\delta(L)}.
\]
Thus $\delta(L)=0$ if and only if every $R_3$-coloring lifts to a
$J_9$-coloring.

For $9_{46}$, Proposition~\ref{prop:knotcounts} gives
$N_3=27$ and $N_9=81$, so $c=3$ and $r=2$.
Exactly $3^{c-r}=3$ of its $R_3$-colorings lift, and each has $3^c=27$
lifts. Every constant $R_3$-coloring has a constant $J_9$-lift, so
the three liftable $R_3$-colorings are exactly the constant ones.
Their subspace has dimension $1$ inside the $3$-dimensional space
$E_D$, giving $\delta(9_{46})=3-1=2$.

\section{Seifert forms, connected sums, and infinite families}
\label{sec:seifert}\label{sec:structure}
The lifting obstruction in Section~\ref{sec:lifting} applies to all
oriented links. Here we restrict to knots and express the coloring
equations using a Seifert matrix. An alternating bilinear form then
describes the lifting obstruction, and its rank gives $\delta(K)$.
We use this description to prove evenness and additivity under
connected sum, before computing the counts for the family in
Theorem~\ref{thm:C}. We also prove that the lifting defect vanishes
for knots admitting matched diagrams.

\subsection{The Alexander module and the constant color}
\label{subsec:alexander-module}
Let $K$ be an oriented knot with a Seifert surface of genus $g$.
Choose oriented curves $a_1,\ldots,a_{2g}$ whose classes form an
integral basis of the first homology of the surface.
Following \cite{HLW}, use the Seifert matrix
$V_{ij}=\operatorname{lk}(a_i,a_j^+)$, where $a_j^+$ is the push-off
in the positive normal direction determined by the surface orientation.
Put
$\Lambda=\mathbb Z[t,t^{-1}]$.
Set $X_K=S^3\setminus K$.
Let $\widetilde X_K\to X_K$ be the infinite cyclic cover defined by
the abelianization $\pi_1(X_K)\to\mathbb Z$.
The Alexander module is
\[
 A_K=H_1(\widetilde X_K;\mathbb Z).
\]
The variable $t$ acts by the deck transformation corresponding to
a positive meridian.

A standard presentation of $A_K$ is given in \cite{HLW}.
There the rows of $V-tV^{\tr}$ record the relations.
We record relations by columns.
Transposing that matrix and multiplying by $-1$ gives
\[
 P(t)=tV-V^{\tr},\qquad
 A_K\cong\Lambda^{2g}/P(t)\Lambda^{2g}.
\]
Thus the images of the standard basis vectors of $\Lambda^{2g}$ generate $A_K$.
The entries in each column of $P(t)$ are the coefficients of one relation.

The parameter in our coloring relation acts as the inverse deck action.
For the left-hand crossing in Figure~\ref{fig:conventions}, choose
compatible positive Wirtinger meridians $\alpha,\beta,\gamma$
corresponding to $a,b,c$. Their relation is
$\alpha=\beta\gamma\beta^{-1}$.
The Fox derivatives of the relator $\alpha\beta\gamma^{-1}\beta^{-1}$,
after sending each meridian to $t$, are $(1,t-1,-t)$.
Thus the Fox relation is $a+(t-1)b-tc=0$.
Replacing $t$ by $t^{-1}$ and multiplying by $t$ gives
\[
 ta+(1-t)b-c=0,
\]
which is our coloring relation $c=ta+(1-t)b$.
Write $A_K^\iota$ for the module with this action, where
$\iota(t)=t^{-1}$. It has the same underlying abelian group as $A_K$,
but $t$ acts as the inverse of its action on $A_K$.
Replacing $t$ by $t^{-1}$ in $P(t)$ gives a presentation of $A_K^\iota$.
Multiplying all relations by the unit $-t$ does not change the module.
Define
\[
 Q(t)=-tP(t^{-1})=tV^{\tr}-V.
\]
Then $Q(t)$ is a $2g\times2g$ column presentation matrix for $A_K^\iota$:
\[
 A_K^\iota\cong\Lambda^{2g}/Q(t)\Lambda^{2g}.
\]

\begin{lemma}\label{lem:module}
Let $M$ be a $\Lambda$-module, with Alexander quandle operation
$x*y=tx+(1-t)y$. Write $\Col_M(K)$ for its coloring module.
For the crossing convention of Figure~\ref{fig:conventions},
\[
 \Col_M(K)\cong M\oplus\Hom_\Lambda(A_K^\iota,M).
\]
The first factor records the color on one chosen arc.
The second records the reduced coloring obtained by subtracting that
color from every arc. Its color on the chosen arc is zero.
Write $\Col_M^0(K)$ for the module of colorings that vanish on this arc.
In particular, if $M$ is finite, then
\begin{equation}\label{eq:seifertcolorings}
 |\Col_M(K)|
 =|M|\,\bigl|\ker(tV-V^{\tr}:M^{2g}\longrightarrow M^{2g})\bigr|.
\end{equation}
\end{lemma}
\begin{proof}
Let $B_D$ be the $\Lambda$-module generated by the arcs of a diagram $D$ of $K$,
with the coloring relations of Figure~\ref{fig:conventions}.
We use the standard identification
$\Col_M(K)\cong\Hom_\Lambda(B_D,M)$; see \cite{Traldi}.
With the action convention established above, Crowell's module sequence
\cite{Crowell} gives
\[
 0\longrightarrow A_K^\iota\longrightarrow B_D
 \xrightarrow{\eta}\Lambda\longrightarrow0.
\]
With our action convention, Crowell's map takes values in the ideal
$(t^{-1}-1)\Lambda$. It sends each arc generator to $t^{-1}-1$.
The isomorphism $(t^{-1}-1)u\mapsto u$ identifies this ideal with $\Lambda$.
Thus the resulting map $\eta$ sends every arc generator to $1$.

Choose an arc generator $x_0\in B_D$.
Since $\eta(x_0)=1$, every $b\in B_D$ has the unique decomposition
\[
 b=\eta(b)x_0+\bigl(b-\eta(b)x_0\bigr),
 \qquad b-\eta(b)x_0\in\ker\eta.
\]
Thus $B_D=\Lambda x_0\oplus\ker\eta$, and the exact sequence
identifies $\ker\eta$ with $A_K^\iota$.
A homomorphism $h:B_D\to M$ is therefore determined by
$m=h(x_0)$ and $f=h|_{\ker\eta}$.
Conversely, any such pair defines a homomorphism by
\[
 h(b)=\eta(b)m+f\bigl(b-\eta(b)x_0\bigr).
\]
These constructions are inverse to each other and give
\[
 \Col_M(K)\cong M\oplus\Hom_\Lambda(A_K^\iota,M).
\]
For an arc generator $x$, we have $f(x-x_0)=h(x)-m$.
Hence $m$ is the color on the chosen arc, and $f$ records the
coloring obtained by subtracting $m$ from every arc color.

We next describe the homomorphisms $A_K^\iota\to M$ explicitly.
Under the presentation $A_K^\iota\cong\Lambda^{2g}/Q(t)\Lambda^{2g}$,
let $\bar e_i$ denote the image of the $i$-th standard basis vector.
The $j$-th column of $Q(t)$ gives the relation
$\sum_{i=1}^{2g}Q(t)_{ij}\bar e_i=0$.
A homomorphism $f:A_K^\iota\to M$ is determined by the values
$z_i=f(\bar e_i)$, since the $\bar e_i$ generate $A_K^\iota$.
These values must satisfy every relation.
Writing $z=(z_1,\ldots,z_{2g})^{\tr}$ and applying $f$ to the
$j$-th relation gives
\[
 0=\sum_{i=1}^{2g}Q(t)_{ij}z_i
   =\bigl(Q(t)^{\tr}z\bigr)_j.
\]
Thus $Q(t)^{\tr}z=0$, so evaluation on the generators defines a
$\Lambda$-linear map
\[
 \begin{aligned}
 \Phi:\Hom_\Lambda(A_K^\iota,M)
 &\longrightarrow\ker\bigl(Q(t)^{\tr}:M^{2g}\longrightarrow M^{2g}\bigr),\\
 f&\longmapsto\bigl(f(\bar e_1),\ldots,f(\bar e_{2g})\bigr)^{\tr}.
 \end{aligned}
\]
It is injective because two homomorphisms with the same values on
all generators are equal.

Conversely, take $z=(z_1,\ldots,z_{2g})^{\tr}$ with $Q(t)^{\tr}z=0$.
Every element of $A_K^\iota$ can be written as
$\sum_i\lambda_i\bar e_i$, with $\lambda_i\in\Lambda$.
Define
\[
 f_z\left(\sum_{i=1}^{2g}\lambda_i\bar e_i\right)
 =\sum_{i=1}^{2g}\lambda_i z_i.
\]
We must check that this value does not depend on the chosen coefficients.
If two coefficient columns $\lambda,\mu\in\Lambda^{2g}$ represent
the same element of $A_K^\iota$, then
$\lambda-\mu=Q(t)w$ for some $w\in\Lambda^{2g}$.
Hence
\[
 \sum_{i=1}^{2g}(\lambda_i-\mu_i)z_i
 =\sum_{j=1}^{2g}w_j\left(\sum_{i=1}^{2g}Q(t)_{ij}z_i\right)
 =0.
\]
Thus $f_z$ is well defined, and its formula shows that it is
$\Lambda$-linear. Moreover, $f_z(\bar e_i)=z_i$, so
$\Phi(f_z)=z$. Conversely, $f_{\Phi(f)}=f$ because the two maps
agree on every generator. This proves that $\Phi$ is an isomorphism:
\[
 \Hom_\Lambda(A_K^\iota,M)
 \cong\ker\bigl(Q(t)^{\tr}:M^{2g}\longrightarrow M^{2g}\bigr).
\]
The right-hand side therefore records the possible generator values
of the homomorphisms on the left-hand side.
For fixed choices, a homomorphism of coefficient modules acts by
postcomposition on $f$ and entrywise on $z$.
Hence these identifications commute with coefficient homomorphisms.
Since $Q(t)^{\tr}=tV-V^{\tr}$, taking cardinalities for finite $M$
proves \eqref{eq:seifertcolorings}.
\end{proof}

\subsection{The alternating form}
Reduce the entries of $V$ modulo $3$, keeping the notation $V$.
All vector spaces and bilinear forms in this subsection are over $\F$.
Set
\[
 W=V+V^{\tr},\qquad
 E_K=\ker(W:\F^{2g}\longrightarrow\F^{2g}),\qquad d=\dim E_K.
\]
For $x,y\in E_K$, define
\[
 \beta_K(x,y)=x^{\tr}Vy.
\]

We first explain the matrix and rank of this form.
Choose a basis $z_1,\ldots,z_d$ of $E_K$.
Each $z_i$ is a column with $2g$ entries, so
\[
 Z=(z_1\ \cdots\ z_d)
\]
is a $2g\times d$ matrix.
The matrix of $\beta_K$ in this basis is the $d\times d$ matrix $B$
whose $(i,j)$ entry is its value on the two basis vectors:
\[
 B_{ij}=\beta_K(z_i,z_j)=z_i^{\tr}Vz_j.
\]
The $(i,j)$ entry of $Z^{\tr}VZ$ is also $z_i^{\tr}Vz_j$.
Therefore $B=Z^{\tr}VZ$.
More explicitly, every $x,y\in E_K$ has unique coordinate columns
$\xi,\eta\in\F^d$ such that
\[
 x=\sum_i\xi_i z_i=Z\xi,\qquad
 y=\sum_j\eta_j z_j=Z\eta.
\]
Bilinearity gives
\[
 \beta_K(x,y)
 =\sum_{i,j}\xi_i B_{ij}\eta_j
 =\xi^{\tr}B\eta.
\]
Thus the rank of the form is the matrix rank
$\rank\beta_K=\rank B=\rank(Z^{\tr}VZ)$.
A change of basis replaces $B$ by $U^{\tr}BU$ for an invertible matrix
$U$, so this rank does not depend on the chosen basis.

The right radical consists of the vectors that pair to zero with
every vector of $E_K$:
\[
 \rad\beta_K
 =\{v\in E_K:\beta_K(x,v)=0\text{ for every }x\in E_K\}.
\]
If $v=Z\eta$, it is enough to test this condition on the basis vectors
$z_i$. Their pairings are
$\beta_K(z_i,v)=(B\eta)_i$.
Hence $v\in\rad\beta_K$ exactly when $B\eta=0$.
The coordinate isomorphism $\eta\mapsto Z\eta$ therefore identifies
$\ker B$ with $\rad\beta_K$.
Applying the rank--nullity theorem to $B:\F^d\to\F^d$ gives
\[
 \rank\beta_K
 =d-\dim\ker B
 =d-\dim\rad\beta_K.
\]

\begin{theorem}\label{thm:seifert}
The form $\beta_K$ is alternating. If $r=\rank\beta_K$, then
\begin{equation}\label{eq:seifertcounts}
 N_3(K)=3^{1+d},\qquad
 N_9(K)=3^{2+2d-r},\qquad
 \delta(K)=r.
\end{equation}
In particular, $0\leq\delta(K)\leq d$ and $\delta(K)$ is even.
\end{theorem}
\begin{proof}
\emph{The alternating property.}
For $x,y\in E_K$, we have
\[
 \beta_K(x,y)+\beta_K(y,x)
 =x^{\tr}(V+V^{\tr})y=x^{\tr}Wy=0.
\]
Here $Wy=0$ because $y\in E_K$.
Taking $y=x$ gives $2\beta_K(x,x)=0$.
Since $2$ is invertible in $\F$, we obtain $\beta_K(x,x)=0$.
Thus $\beta_K$ is alternating.

\smallskip
\emph{Counting the $R_3$-colorings.}
Fix the chosen arc used in Lemma~\ref{lem:module}.
For $R_3$, the parameter is $t=-1$, so the matrix in that lemma is
\[
 tV-V^{\tr}=-V-V^{\tr}=-W.
\]
The kernel is $E_K$, so there are $3^d$ reduced $R_3$-colorings.
Lemma~\ref{lem:module} gives three independent choices for the
color on the chosen arc. Consequently,
\[
 N_3(K)=3\cdot3^d=3^{1+d}.
\]

\smallskip
\emph{The equations for a reduced $J_9$-coloring.}
For $J_9$, the parameter is $\theta=-1+\eps$.
By Lemma~\ref{lem:module}, reduced $J_9$-colorings correspond to the kernel of
\[
 \theta V-V^{\tr}=-W+\eps V:
 \R^{2g}\longrightarrow\R^{2g}.
\]
Every vector in $\R^{2g}$ has a unique expression $v+\eps a$ with
$v,a\in\F^{2g}$.
Since $\eps^2=0$,
\[
 (-W+\eps V)(v+\eps a)
 =-Wv+\eps(Vv-Wa).
\]
The elements $1,\eps$ form a basis of $\R$ over $\F$.
This expression is zero exactly when both coefficients are zero.
Therefore
\[
 v+\eps a\in\ker(-W+\eps V)
 \quad\Longleftrightarrow\quad
 Wv=0\ \text{and}\ Wa=Vv.
\]
We must first choose $v\in E_K$.
Once $v$ is fixed, $a$ is the unknown vector, and the remaining
equation to solve is $Wa=Vv$.

These vectors also describe lifts of reduced colorings.
Use the same chosen arc and the same Alexander-module presentation
for both targets.
Under the coefficient map $\rho:\R\to\F$, $\rho(v+\eps a)=v$.
The correspondence in Lemma~\ref{lem:module} commutes with this reduction:
a homomorphism $f:A_K^\iota\to\R$ is replaced by $\rho\circ f$.
Its values on the same generators are reduced entry by entry.
Thus a solution $a$ gives a reduced $J_9$-coloring lifting the
reduced $R_3$-coloring represented by $v$.

\smallskip
\emph{When the equation for $a$ has a solution.}
By definition, $\im W$ consists of the vectors $Wa$.
For a fixed $v\in E_K$, the equation $Wa=Vv$ therefore has a solution
if and only if $Vv\in\im W$.
We now describe this image.

Take the orthogonal complement of $E_K$ in the full space $\F^{2g}$,
using the standard dot product:
\[
 E_K^\perp
 =\{z\in\F^{2g}:x^{\tr}z=0\text{ for every }x\in E_K\}.
\]
Since $W$ is symmetric, for $x\in E_K$ and $a\in\F^{2g}$ we have
\[
 x^{\tr}Wa=(Wx)^{\tr}a=0.
\]
Every vector in $\im W$ is therefore in $E_K^\perp$.

To compare dimensions, use the basis matrix $Z$ defined above.
A vector $z$ lies in $E_K^\perp$ exactly when $z_i^{\tr}z=0$
for all $i$, or equivalently $Z^{\tr}z=0$.
Thus
\[
 E_K^\perp=\ker(Z^{\tr}:\F^{2g}\longrightarrow\F^d).
\]
The rows of $Z^{\tr}$ are linearly independent because the columns
$z_1,\ldots,z_d$ of $Z$ form a basis of $E_K$.
Hence $\rank Z^{\tr}=d$, and the rank--nullity theorem gives
\[
 \dim E_K^\perp=2g-\rank Z^{\tr}=2g-d.
\]
Applying the same theorem to $W:\F^{2g}\to\F^{2g}$ gives
\[
 \dim\im W=2g-\dim\ker W=2g-d.
\]
The inclusion and equality of dimensions now imply
$\im W=E_K^\perp$.

It follows that $Wa=Vv$ has a solution exactly when
$Vv\in E_K^\perp$, that is,
\[
 x^{\tr}Vv=\beta_K(x,v)=0
 \qquad\text{for every }x\in E_K.
\]
This is precisely the defining condition for $v\in\rad\beta_K$.
Thus a vector $v\in E_K$ outside the radical has no lift.
For each $v\in\rad\beta_K$, at least one solution $a_0$ exists.

\smallskip
\emph{Counting the solutions and the $J_9$-colorings.}
There are two choices to make.
First choose $v\in\rad\beta_K$.
Since this space has dimension $d-r$, there are $3^{d-r}$ choices.
Next, fix this $v$ and one solution $a_0$ of $Wa=Vv$.
Another vector $a$ solves the same equation if and only if
\[
 W(a-a_0)=0
 \quad\Longleftrightarrow\quad a-a_0\in E_K.
\]
Hence all solutions, and only those solutions, are
\[
 a=a_0+h,\qquad h\in E_K.
\]
There are $3^d$ choices for $h$, giving $3^d$ distinct choices for $a$.

Each such pair $(v,a)$ gives one vector $v+\eps a$ in the kernel.
Conversely, every vector in the kernel gives exactly one such pair,
because its constant and $\eps$ coefficients are unique.
Thus the number of reduced $J_9$-colorings is
\[
 |\Col_{J_9}^0(K)|
 =|\ker(-W+\eps V)|
 =3^{d-r}\cdot3^d
 =3^{2d-r}.
\]
By Lemma~\ref{lem:module}, each reduced coloring can be combined
with any of the nine colors on the chosen arc.
Therefore the total coloring count is
\[
 N_9(K)=9\cdot3^{2d-r}=3^{2+2d-r}.
\]
Substituting the two counts into the definition of the defect gives
\[
 \delta(K)=2\log_3N_3(K)-\log_3N_9(K)
 =2(1+d)-(2+2d-r)=r.
\]

\smallskip
\emph{Why the defect is even.}
We have proved $\delta(K)=r$, so it remains to show that the rank
of $\beta_K$ is even.
If $\beta_K=0$, then $r=0$.
Otherwise choose $x,y\in E_K$ with
$\lambda=\beta_K(x,y)\ne0$.
These vectors are linearly independent, since a multiple of $x$
pairs to zero with $x$.
On $H=\operatorname{span}\{x,y\}$, the matrix of the form is
\[
 \begin{pmatrix}0&\lambda\\-\lambda&0\end{pmatrix}.
\]
Its determinant is $\lambda^2\ne0$, so this part of the form has rank $2$.

Define the complement using $\beta_K$:
\[
 H^{\perp_\beta}
 =\{z\in E_K:\beta_K(x,z)=\beta_K(y,z)=0\}.
\]
For any $z\in E_K$, put
\[
 h=-\frac{\beta_K(y,z)}{\lambda}x
   +\frac{\beta_K(x,z)}{\lambda}y\in H.
\]
Direct substitution gives
$\beta_K(x,h)=\beta_K(x,z)$ and
$\beta_K(y,h)=\beta_K(y,z)$.
Thus $z-h\in H^{\perp_\beta}$, so every vector is a sum of vectors
from $H$ and $H^{\perp_\beta}$.
The intersection is zero: if $ax+by\in H^{\perp_\beta}$, then
$b\lambda=0$ and $-a\lambda=0$, forcing $a=b=0$.
We have therefore proved
\[
 E_K=H\oplus H^{\perp_\beta}.
\]
The pairings from $H$ to $H^{\perp_\beta}$ vanish by definition,
and the reverse pairings vanish because $\beta_K$ is alternating.
In a basis adapted to this direct sum, the matrix is block diagonal.
Its first block has rank $2$, so
\[
 r=2+\rank\bigl(\beta_K|_{H^{\perp_\beta}}\bigr).
\]
The restricted form is again alternating.
If it is nonzero, repeat the same construction on that smaller space.
Each step removes a block of rank $2$ and reduces the dimension by $2$.
The process ends with a zero form, which has rank $0$.
Thus $r$ is a sum of copies of $2$, and $\delta(K)=r$ is even.
Finally, $0\leq r\leq d$ because $r$ is the rank of a $d\times d$ matrix.
\end{proof}

For a diagram $D$ of $K$, the space $E_D$ in
Section~\ref{sec:lifting} consists of all $R_3$-colorings, whereas
$E_K$ parametrizes the reduced colorings. Lemma~\ref{lem:module} gives
\[
 E_D\cong\F\oplus E_K,\qquad c=1+d.
\]

For the unknot, take a Seifert disk, so $g=0$ and the Seifert matrix
is empty. Then $E_K=0$ and $d=r=0$.
The formulas give $N_3=3$, $N_9=9$, and $\delta=0$, as expected.

\begin{example}[Equal Alexander polynomials and different defects]
\label{ex:same-alexander}
Lemma~1(2)--(4) of Friedl and Teichner~\cite{FT} gives the integral
Seifert matrices
\[
 V_{9_{46}}=\begin{pmatrix}0&2\\1&0\end{pmatrix},
 \qquad
 V_{6_1}=\begin{pmatrix}-2&2\\1&0\end{pmatrix}.
\]
We compute $W$, $E_K$, and $\beta_K$ from their reductions modulo $3$,
as in Theorem~\ref{thm:seifert}.

For $9_{46}$, the matrix $W=V+V^{\tr}$ is zero modulo $3$.
Thus $E_{9_{46}}=\F^2$, $d=2$, and we may take $Z=I_2$.
The matrix of $\beta_{9_{46}}$ is then
$Z^{\tr}VZ=V$, whose determinant is $-2=1$ in $\F$.
Hence $r=2$.

For $6_1$, reduction modulo $3$ gives
\[
 W=\begin{pmatrix}2&0\\0&0\end{pmatrix},
 \qquad E_{6_1}=\F e_2,\qquad e_2=(0,1)^{\tr}.
\]
Here $d=1$, and the matrix of the restricted form is
$e_2^{\tr}V_{6_1}e_2=0$, so $r=0$.
Substitution into \eqref{eq:seifertcounts} gives
\[
\begin{array}{c|cc|ccc}
 K & d&r&N_3(K)&N_9(K)&\delta(K)\\ \hline
 9_{46}&2&2&27&81&2\\
 6_1&1&0&9&81&0
\end{array}
\]
The first row agrees with the diagram computation in
Proposition~\ref{prop:knotcounts}.

For both displayed integral matrices,
\[
 \det(tV-V^{\tr})=-(2t-1)(t-2).
\]
Thus the two knots have the same Alexander polynomial, up to units,
but different lifting defects. Their $N_9$-counts also agree, so
$N_9$ alone does not determine the defect.
Alexander quandle invariants distinguishing knots with the same
Alexander polynomial were studied by Benedetti and Frigerio
\cite{BF}; this example exhibits that distinction for the present
lifting defect.
\end{example}

\subsection{Additivity under connected sum}
\begin{theorem}\label{thm:additivity}
For oriented knots $K_1,K_2$,
\[
 \delta(K_1\#K_2)=\delta(K_1)+\delta(K_2).
\]
\end{theorem}
\begin{proof}
Let $V_i$ be a Seifert matrix for $K_i$, $i=1,2$.
By Lemma~1.12(1) of~\cite{Conway}, we may use
$V=V_1\oplus V_2$ as a Seifert matrix for $K_1\#K_2$.

Apply the definitions of $W$, $E_K$, and $\beta_K$ to these matrices.
They give
\[
\begin{aligned}
 W_{K_1\#K_2}&=W_{K_1}\oplus W_{K_2},\\
 E_{K_1\#K_2}&=E_{K_1}\oplus E_{K_2},\\
 \beta_{K_1\#K_2}&=\beta_{K_1}\oplus\beta_{K_2}.
\end{aligned}
\]
Write $r_i=\rank\beta_{K_i}$.
Choose a basis of each $E_{K_i}$, and let $B_i$ be the matrix of
$\beta_{K_i}$ in that basis. In the combined basis, the matrix of
$\beta_{K_1\#K_2}$ is
\[
 \begin{pmatrix}B_1&0\\0&B_2\end{pmatrix},
\]
whose rank is $r_1+r_2$.
Theorem~\ref{thm:seifert} therefore gives
$\delta(K_1\#K_2)=r_1+r_2=\delta(K_1)+\delta(K_2)$.
\end{proof}

\begin{proof}[Proof of Theorem~\ref{thm:B}]
This follows from Theorems~\ref{thm:seifert} and~\ref{thm:additivity}.
\end{proof}

\begin{proof}[Another proof of Theorem~\ref{thm:additivity}]
We prove additivity directly from quandle coloring counts, using two
lemmas of Clark, Saito and Vendramin.

A pair $(K,Q)$ is called \emph{end-monochromatic} if every $Q$-coloring
of the $1$-tangle obtained by cutting a small open arc from $K$ assigns
the same color to its two endpoints.
By Lemma~4.4 of~\cite{CSV}, if $Q$ is faithful, then $(K,Q)$ is
end-monochromatic for every knot $K$.
By Lemma~4.6 of~\cite{CSV}, if $Q$ is finite and homogeneous, and at least
one of $(K_1,Q)$ and $(K_2,Q)$ is end-monochromatic, then
\[
 |\Col_Q(K_1\#K_2)|
 =\frac{|\Col_Q(K_1)|\,|\Col_Q(K_2)|}{|Q|}.
\]
Here the connected sum respects the orientations of both knots.
We check these hypotheses for $Q=R_3$ and $Q=J_9$.

Both quandles are finite, of orders $3$ and $9$, respectively.
To verify homogeneity, observe that each translation $x\mapsto x+a$
is a quandle automorphism: it has inverse $x\mapsto x-a$ and satisfies
\[
 (x+a)*(y+a)=(x*y)+a.
\]
For any two colors $x,z$, translation by $z-x$ sends $x$ to $z$.
Thus the automorphism group acts transitively, which is the
homogeneity condition in Lemma~4.6.

Faithfulness means that different colors determine different right
translations. Write $R_y(x)=tx+(1-t)y$, where $t=-1$ for $R_3$
and $t=\theta$ for $J_9$.
If $R_y=R_z$, evaluating at $x=0$ gives
\[
 (1-t)(y-z)=0.
\]
For $R_3$, $1-t=2$ is invertible.
For $J_9$, $1-\theta=2-\eps$ has inverse $2+\eps$, since
$(2-\eps)(2+\eps)=1$ in $\R$.
Hence $y=z$, proving faithfulness.

Lemma~4.4 therefore shows that both $(K_1,Q)$ and $(K_2,Q)$ are
end-monochromatic for each of the two targets. In particular, the
requirement that at least one factor be end-monochromatic is satisfied.
Together with finiteness and homogeneity, this verifies every
hypothesis of Lemma~4.6. Applying its formula with $|R_3|=3$ and
$|J_9|=9$ gives
\begin{equation}\label{eq:connectedcounts}
 \begin{aligned}
 N_3(K_1\#K_2)&=\frac{N_3(K_1)N_3(K_2)}{3},\\
 N_9(K_1\#K_2)&=\frac{N_9(K_1)N_9(K_2)}{9}.
 \end{aligned}
\end{equation}
Finally, substitute these counts into the definition of $\delta$:
\[
 \begin{aligned}
 \delta(K_1\#K_2)
 &=\delta(K_1)+\delta(K_2)-2\log_3 3+\log_3 9\\
 &=\delta(K_1)+\delta(K_2).
 \end{aligned}
\]
\end{proof}

\subsection{The connected-sum family}
\begin{proof}[Proof of Theorem~\ref{thm:C} and its consequences]
Recall that $K_s$ is the connected sum of $s\geq1$ copies of $9_{46}$.
Proposition~\ref{prop:knotcounts} gives
$N_3(9_{46})=27$ and $N_9(9_{46})=81$.
Applying \eqref{eq:connectedcounts} at the $s-1$
connected sums gives
\[
 N_3(K_s)=\frac{27^s}{3^{s-1}}=3^{2s+1},\qquad
 N_9(K_s)=\frac{81^s}{9^{s-1}}=3^{2s+2}.
\]
Hence
\[
 \delta(K_s)=2(2s+1)-(2s+2)=2s.
\]
By Proposition~\ref{prop:defect}, $\delta$ is unchanged by
$t_3$- and $\bar t_6$-moves and is zero on every unlink.
Since $\delta(K_s)>0$, no $K_s$ is equivalent to an unlink.
Different values of $s$ give different defects, so
$K_s\not\sim K_t$ when $s\ne t$.
Thus the family contains infinitely many distinct knots and
infinitely many equivalence classes under these oriented moves.

For the ordinary $3$-move reduction, forget the orientations.
Take $s$ disjoint standard pretzel diagrams $P(3,3,-3)$.
Choose joining arcs on their exterior closing arcs, away from the
twist boxes. Form their planar connected sum in a chain, without
adding crossings.
The three $3$-moves in Corollary~\ref{cor:separation} take place inside
the twist boxes, so they do not affect the joining arcs.

Apply the three moves in each summand, using $3s$ moves in total.
Each summand becomes three planar circles, with the $s-1$ connections
still in place. Before making these connections there are $3s$ circles.
The summands are joined in a chain. Each connection attaches a new
summand to the preceding ones, so it joins two distinct components
and decreases the number of components by one.
Thus the resulting crossingless diagram has
\[
 3s-(s-1)=2s+1
\]
components and represents $U_{2s+1}$.

By Lemma~\ref{lem:mirror}, reversing the orientation or taking the
mirror does not change the coloring counts.
Reflecting the three moves in each summand gives the corresponding
reduction for the mirror knot.
Thus the same conclusions hold for the family of mirrors.

Finally, Theorem~\ref{thm:seifert} shows that every knot has a
nonnegative even defect.
The unknot realizes $0$, and $K_s$ realizes $2s$ for each $s\geq1$.
The set of values on knots is therefore exactly $\{0,2,4,\ldots\}$.
\end{proof}

\subsection{Knots admitting matched diagrams}\label{subsec:matched}
Following Duzhin and Shkolnikov \cite{DS}, a knot diagram is
called \emph{matched} if all its crossings can be partitioned into
pairs of the local types shown in Figure~1 of~\cite{DS}.
A knot admitting a matched diagram is called \emph{bipartite}.

\begin{corollary}[Consequence of Duzhin--Shkolnikov]\label{cor:matched}
If a knot $K$ admits a matched diagram, then $\delta(K)=0$.
\end{corollary}
\begin{proof}
By Lemma~1 of~\cite{DS}, the knot $K$ has a Seifert matrix
\[
 V=\begin{pmatrix}S&0\\I_n&T\end{pmatrix},
 \qquad S=S^{\tr},\quad T=T^{\tr},
\]
where $S,T$ are $n\times n$ integer matrices.
The Seifert-pairing convention in that lemma agrees with
Subsection~\ref{subsec:alexander-module}.
Reduce the entries modulo $3$, retaining the same notation.
The matrix used in Theorem~\ref{thm:seifert} is then
\[
 W=V+V^{\tr}
   =\begin{pmatrix}-S&I_n\\I_n&-T\end{pmatrix}.
\]
Write a vector in $\F^{2n}$ as $x=\binom{u}{v}$, with $u,v\in\F^n$.
The equation $Wx=0$ says $v=Su$ and $Tv=u$. Hence
\[
 E_K=\left\{\binom{u}{Su}:u\in\F^n,\ TSu=u\right\}.
\]
Take $x=\binom{u}{Su}$ and $y=\binom{w}{Sw}$ in $E_K$.
Since $TSw=w$, we have
\[
 Vy=\binom{Sw}{w+TSw}=\binom{Sw}{2w}.
\]
Using $S^{\tr}=S$, we obtain
\[
 \beta_K(x,y)=x^{\tr}Vy
 =u^{\tr}Sw+2(Su)^{\tr}w
 =3u^{\tr}Sw=0\quad\text{in }\F.
\]
Thus $\beta_K$ is the zero form, and Theorem~\ref{thm:seifert} gives
$\delta(K)=\rank\beta_K=0$.
\end{proof}

By Proposition~\ref{prop:defect}, $\delta$ is invariant under $\sim$.
Consequently, no knot with positive defect is $\sim$-equivalent
to a knot admitting a matched diagram.
In particular, no knot $\sim$-equivalent to $K_s$, $s\geq1$,
admits a matched diagram, since $\delta(K_s)=2s$ by
Theorem~\ref{thm:C}.
Duzhin and Shkolnikov already proved that $9_{46}$ has no matched
diagram \cite{DS}.
We do not claim that $\delta(K)=0$ guarantees the existence of
a matched diagram.

\begin{remark}\label{rem:nonexhaustive}
The knots $K_s$ do not represent all $\sim$-equivalence classes
of knots, even if we set $K_0=U_1$.
By Example~\ref{ex:same-alexander}, $N_3(6_1)=9$.
On the other hand, Theorem~\ref{thm:C} and \eqref{eq:unlink} give
$N_3(K_s)=3^{2s+1}$ for $s\geq0$, which is never $9$.
Since $N_3$ is invariant under $\sim$, $6_1$ is not equivalent
to any $K_s$.
Thus realizing every possible value of $\delta$ does not mean
realizing every equivalence class.
\end{remark}

\section*{Acknowledgments}
The author acknowledges assistance from GPT-6 (OpenAI) with programming and language editing. All results were checked by the author. The author takes full responsibility for the content of this paper.


\begin{thebibliography}{99}
\bibitem{BF}
R.~Benedetti and R.~Frigerio,
\emph{Alexander quandle lower bounds for link genera},
J. Knot Theory Ramifications \textbf{21} (2012), no.~8, 1250076.
\href{https://doi.org/10.1142/S0218216512500769}{doi:10.1142/S0218216512500769}.

\bibitem{CESY}
W.~E. Clark, M.~Elhamdadi, M.~Saito and T.~Yeatman,
\emph{Quandle colorings of knots and applications},
J. Knot Theory Ramifications \textbf{23} (2014), no.~6, 1450035.
\href{https://doi.org/10.1142/S0218216514500357}{doi:10.1142/S0218216514500357}.

\bibitem{CSV}
W.~E. Clark, M.~Saito and L.~Vendramin,
\emph{Quandle coloring and cocycle invariants of composite knots and abelian extensions},
J. Knot Theory Ramifications \textbf{25} (2016), no.~5, 1650024.
\href{https://doi.org/10.1142/S0218216516500243}{doi:10.1142/S0218216516500243}.

\bibitem{Conway}
A.~Conway, \emph{Algebraic concordance and Casson--Gordon invariants},
lecture notes, revised 2023.
\href{https://drive.google.com/file/d/1hC5Bq1GxUmDMR7VGkJmTT_QnuLq-t9vt/view}{Author's notes}.
Accessed 19 October 2026.

\bibitem{Crowell}
R.~H. Crowell, \emph{Corresponding group and module sequences},
Nagoya Math. J. \textbf{19} (1961), 27--40.
\href{https://doi.org/10.1017/S0027763000002373}{doi:10.1017/S0027763000002373}.

\bibitem{DP}
M.~K. D\k{a}bkowski and J.~H. Przytycki,
\emph{Burnside obstructions to the Montesinos--Nakanishi $3$-move conjecture},
Geom. Topol. \textbf{6} (2002), 355--360.
\href{https://doi.org/10.2140/gt.2002.6.355}{doi:10.2140/gt.2002.6.355}.

\bibitem{DS}
S.~Duzhin and M.~Shkolnikov,
\emph{Bipartite knots}, Fund. Math. \textbf{225} (2014), no.~1, 95--102.
\href{https://doi.org/10.4064/fm225-1-6}{doi:10.4064/fm225-1-6}.

\bibitem{FK}
T.~Fiedler and V.~Kurlin,
\emph{Recognizing trace graphs of closed braids},
Osaka J. Math. \textbf{47} (2010), no.~4, 885--909.
\href{https://doi.org/10.18910/7426}{doi:10.18910/7426}.

\bibitem{FT}
S.~Friedl and P.~Teichner,
\emph{Clarification to ``New topologically slice knots''},
authors' note, 13 December 2007.
\href{https://friedl.app.uni-regensburg.de/papers/topslice_second_correction.pdf}{Author's version}.

\bibitem{HLW}
J.~Hom, T.~Lidman and L.~Watson,
\emph{The Alexander module, Seifert forms, and categorification},
J. Topol. \textbf{10} (2017), no.~1, 22--100.
\href{https://doi.org/10.1112/topo.12001}{doi:10.1112/topo.12001}.

\bibitem{Joyce}
D.~Joyce, \emph{A classifying invariant of knots, the knot quandle},
J. Pure Appl. Algebra \textbf{23} (1982), no.~1, 37--65.
\href{https://doi.org/10.1016/0022-4049(82)90077-9}{doi:10.1016/0022-4049(82)90077-9}.

\bibitem{Kirby}
R.~Kirby (ed.), \emph{Problems in low-dimensional topology},
in \emph{Geometric Topology} (Athens, GA, 1993),
W.~H. Kazez (ed.), AMS/IP Stud. Adv. Math. \textbf{2.2},
Amer. Math. Soc. and International Press, 1997, 35--473.

\bibitem{Atlas}
The Knot Atlas, \emph{$9_{46}$} and \emph{Braid Representatives},
\url{https://katlas.org/wiki/9_46},
\url{https://katlas.org/wiki/Braid_Representatives}.
Accessed 19 October 2026.

\bibitem{PT}
J.~H. Przytycki and T.~Tsukamoto,
\emph{The fourth skein module and the Montesinos--Nakanishi conjecture for $3$-algebraic links},
J. Knot Theory Ramifications \textbf{10} (2001), no.~7, 959--982.
\href{https://doi.org/10.1142/S0218216501001281}{doi:10.1142/S0218216501001281}.

\bibitem{Traldi}
L.~Traldi, \emph{Multivariate Alexander colorings},
J. Knot Theory Ramifications \textbf{27} (2018), no.~14, 1850076.
\href{https://doi.org/10.1142/S0218216518500761}{doi:10.1142/S0218216518500761}.

\end{thebibliography}
\end{document}